\documentclass[oneside,11pt]{amsart}

\usepackage[T1]{fontenc}

\usepackage[small]{caption}

\usepackage{amssymb}
\usepackage{amsxtra}
\usepackage{hyperref}

\usepackage[arrow, tips, matrix]{xy}
\SelectTips{cm}{12}
\CompileMatrices

\newtheorem{thm}{Theorem}[section]

\newtheorem*{RelativeCannon}{Relative Cannon Conjecture}
\newtheorem{prop}[thm]{Proposition}
\newtheorem{lem}[thm]{Lemma}
\newtheorem{cor}[thm]{Corollary}
\newtheorem{conj}[thm]{Conjecture}

\theoremstyle{definition}
\newtheorem{defn}[thm]{Definition}

\renewcommand{\bar}[1]{\overline{#1}}
\newcommand{\boundary}{\partial}

\newcommand{\set}[2]{\{\,{#1} \mid {#2} \,\}}
\newcommand{\bigset}[2]{ \bigl\{ \, {#1} \bigm| {#2} \, \bigr\} }
\renewcommand{\emptyset}{\varnothing}

\newcommand{\field}[1]{\mathbb{#1}}
\newcommand{\Z}{\field{Z}}
\newcommand{\R}{\field{R}}

\newcommand{\E}{\field{E}}

\newcommand{\C}{\field{C}}
\newcommand{\Hyp}{\field{H}}

\renewcommand{\P}{\mathcal{P}}

\DeclareMathOperator{\Isom}{Isom}

\DeclareMathOperator{\Stab}{Stab}

\DeclareMathOperator{\Homeo}{Homeo}
\DeclareMathOperator{\PSL}{PSL}
\DeclareMathOperator{\Ends}{Ends}

\usepackage{combelow}
\newcommand{\Drutu}{Dru{\cb{t}}u}

\newcommand{\Caratheodory}{Cara\-th\'{e}o\-dory}

\newcommand{\Haissinsky}{Ha\"{i}ssin\-sky}
\newcommand{\Sierpinski}{Sier\-pi{\'n}\-ski}
\newcommand{\Kerekjarto}{Ke\-r\'{e}k\-j\'{a}r\-t\'{o}}

\usepackage{color}
\usepackage{pinlabel}
\usepackage{todonotes}

\usepackage{ifthen}

\newcommand{\showcomments}{yes}

\newsavebox{\commentbox}
{\ifthenelse{\equal{\showcomments}{yes}}%
{\footnotemark
    \begin{lrbox}{\commentbox}
    \begin{minipage}[t]{1.25in}\raggedright\sffamily\upshape\tiny
    \footnotemark[\arabic{footnote}]}
{\begin{lrbox}{\commentbox}}}%
{\ifthenelse{\equal{\showcomments}{yes}}%
{\end{minipage}\end{lrbox}\marginpar{\usebox{\commentbox}}}
{\end{lrbox}}}

\title{Relatively hyperbolic groups with planar boundaries}

\author[G.C.~Hruska]{G.~Christopher Hruska}
\address{Department of Mathematical Sciences\\
University of Wisconsin--Milwaukee\\
PO Box 413\\
Milwaukee, WI 53211\\
USA}
\email{chruska@uwm.edu}

\author[G.S.~Walsh]{Genevieve S.~Walsh }
\address{Department of Mathematics\\
Tufts University\\
Medford, MA 02155\\
USA}
\email{genevieve.walsh@tufts.edu}

\begin{document}

\begin{abstract}
In this article, we prove a version of Martin and Skora's conjecture that convergence groups on the $2$--sphere are covered by Kleinian groups. Given a relatively hyperbolic group pair $(G,\mathcal{P})$ with planar boundary and no \Sierpinski\ carpet or cut points in its limit set, and with $G$ one ended and virtually having no $2$--torsion, we show that $G$ is virtually Kleinian. We also give applications to various versions of the Cannon conjecture and to convergence groups acting on $S^2$. 
\end{abstract}

\maketitle

\section{Introduction} 
\label{sec:Introduction}

Gehring and Martin introduced the notion of a convergence group in \cite{GehringMartin87} in order to describe the topological dynamics of Kleinian groups and, more generally, discrete quasiconformal groups of homeomorphisms of $S^n$.
Later, Freden and Tukia independently proved that any word hyperbolic or relatively hyperbolic group acts as a convergence group on its boundary \cite{Fredenconv,Tukia94}. 
In general, it can be difficult to determine which groups have a particular topological space as their boundary. Very few cases are known.  Famously, a relatively hyperbolic group with boundary $S^0$ must be virtually cyclic \cite{Freudenthal_Ends,Hopf_Ends}, and a relatively hyperbolic group with boundary $S^1$ is always virtually Fuchsian \cite{Tukia_Fuchsian,Gabai92,CassonJungreis94}.  
Whether all hyperbolic groups with boundary $S^2$ are virtually Kleinian is a well-known open problem, known as the Cannon conjecture \cite{Cannon91}.

In this article, we examine relatively hyperbolic groups with planar boundary, establishing that they are virtually Kleinian in many cases.
We explicitly assume, in the following result, that the boundary does not contain a \Sierpinski\ carpet, which rules out the case of groups with $S^2$ boundary.

\begin{thm}
\label{thm:main} 
Let $(G, \mathcal{P})$ be a relatively hyperbolic group pair with planar boundary $\partial(G, \mathcal{P})$.  Assume that $\partial(G, \mathcal{P})$ is connected with no cut points and does not contain a \Sierpinski\ carpet, and that $G$ is one ended and virtually has no $2$--torsion.  Then $G$ is virtually Kleinian. 
\end{thm}

Although the boundary in Theorem~\ref{thm:main} is assumed to be planar, we do not assume that the action on the boundary extends to an action on $S^2$.
There exist groups whose boundaries embed in $S^2$ but for which the action does not extend to $S^2$.  In these cases, one must pass to a finite index subgroup to obtain a Kleinian group \cite{KapovichKleiner00,HruskaStarkTran_DontAct}. Additionally, it is easy to construct examples with cut points where the conclusion of this theorem does not hold, so the cut point hypothesis is necessary (see \cite[\S 4]{HruskaWalsh}). 

Even though $G$ is shown to be virtually Kleinian in Theorem~\ref{thm:main}, the parabolic subgroups of the Kleinian representation may be different from the given ones. For instance, there exist geometrically finite convergence groups on $S^2$ whose peripheral subgroups are closed hyperbolic surface groups. 

\Haissinsky \ \cite{Haissinsky_Invent} proves the special case of Theorem~\ref{thm:main} in which $G$ is hyperbolic and $\mathcal{P}$ is empty.   As in \cite{Haissinsky_Invent}, our proof depends on a hierarchical finiteness theorem of Louder--Touikan \cite{LouderTouikan17}.  In our case, the elementary hierarchy in question involves splittings over finite, parabolic, or loxodromic subgroups, as described in Section~\ref{sec:Combination}.  In \cite{HPWpreprint}, relatively hyperbolic group pairs with Schottky set boundary are studied and many are shown to be covered by Kleinian groups and furthermore classified. 

A conjecture of Martin--Skora states that every convergence group of homeomorphisms of $S^2$ is covered by a discrete subgroup of $\Isom(\Hyp^3)$ \cite[Conj.~6.2]{MartinSkora89}.  A covering of convergence groups is a certain type of equivariant quotient map; see Sections \ref{sec:Convergence} and~\ref{sec:Planar} for definitions and discussion of convergence groups and coverings.  
The following relative version of the Cannon conjecture is a special case of Martin--Skora's covering conjecture.  Note that this conjecture is slightly different than the conjecture studied in \cite{GrovesManningSisto19}. In particular, since there are no restrctions on the peripheral groups, the set of peripheral groups could be empty (and thus this version contains the Cannon conjecture). The peripheral subgroups could also be hyperbolic surface groups and understanding this case requires our work in section \ref{sec:unpinch}. 

\begin{RelativeCannon}
Let $(G,\mathcal{P})$ be a relatively hyperbolic group pair whose Bowditch boundary is homeomorphic to $S^2$. Then $G$ acts properly, isometrically, and geometrically finitely on $\Hyp^3$.
\end{RelativeCannon}

In this paper, we prove a different version of Martin--Skora's covering conjecture: 

\begin{thm}
\label{thm:generalcase}
Suppose $(G,\P)$ is relatively hyperbolic and $G$ has a finite index subgroup containing no element of order two.
Assume the boundary $M=\boundary(G,\P)$ is planar and the action of $G$ on $M$ extends to a faithful convergence group action on $S^2$.

If $M$ does not contain an embedded \Sierpinski\ carpet \textup{(}or if the relative Cannon conjecture is true\textup{)}
then the action of $G$ on $S^2$ is covered by a Kleinian action on $\widehat{\C}$.
If, furthermore, each member of $\P$ is virtually abelian, then the action of $G$ is topologically conjugate to a Kleinian action on $\widehat{\C}$. 
\end{thm} 

The above theorem is a key ingredient in the proof of Theorem~\ref{thm:main}.  However, the hypotheses of Theorem~\ref{thm:generalcase} are more general than those of Theorem~\ref{thm:main}; it applies even if $G$ is multi-ended or if the boundary has cut points. For instance, a relatively hyperbolic group whose planar boundary contains parabolic cut pairs is discussed in \cite[\S 3.4, Example~1]{GaboriauPaulin01} and \cite[Prop.~4.4]{HruskaWalsh}. According to \cite{GaboriauPaulin01}, the action of this group on its boundary extends to $S^2$ but has parabolic subgroups that are virtually free. (Thus, this action is not Kleinian.)
Theorem~\ref{thm:generalcase} implies that this action on $S^2$ is covered by a Kleinian action.

Kleinian groups have only virtually abelian parabolic subgroups.  
In order to prove Theorem~\ref{thm:generalcase}, we establish the following result, which allows us to remove hyperbolic subgroups from the peripheral structure while still acting as a convergence group on $S^2$. The techniques in the proof may be applicable to other situations. 

\begin{thm}
\label{thm:unpinch}
Let $(G, \mathcal{P})$ be a relatively hyperbolic group pair whose action on its planar boundary $M = \partial(G, \mathcal{P})$ extends to a convergence group action on $S^2$. Let $\mathcal{P}'$ be a subset of $\mathcal{P}$ closed under conjugation and containing all of the non-hyperbolic members of $\mathcal{P}$.  Then the boundary $\partial(G, \mathcal{P}')$ is also planar and the action on this boundary extends to a convergence group action on $S^2$ that covers the action of $(G, \mathcal{P})$ on $S^2$.
\end{thm} 

Theorem~\ref{thm:generalcase} has the following consequence relating the Cannon and relative Cannon conjectures:

\begin{thm}
\label{thm:CannonEquivalent}
The Cannon conjecture for torsion-free groups is equivalent to the relative Cannon conjecture for torsion-free groups. 
\end{thm}

The proof of Theorem~\ref{thm:CannonEquivalent} depends on a result of Groves--Manning--Sisto \cite{GrovesManningSisto19} and can be found in Section~\ref{sec:Applications}, where we also relate the Cannon conjecture to the geometrically finite case of Martin--Skora's covering conjecture.

As mentioned above, the proof of Theorem~\ref{thm:main} depends on Theorem~\ref{thm:generalcase}. However, it also uses the following theorem, which deals with graphs of virtually Kleinian groups in the case that the fundamental group of the graph of groups is not known to act on $S^2$. The proof of this theorem depends on Wise's virtually compact special theorem \cite{Wise_QCHierarchy} as well as a residual finiteness result due to Huang--Wise \cite{HuangWise_Stature}

\begin{thm}
\label{thm:puttogether}
Let $G$ be the fundamental group of a graph of groups where each vertex group is virtually Kleinian.  
Let $(M_v,P_v)$ be the pared $3$--manifold corresponding to the finite index torsion-free Kleinian subgroup $H_v$ of the vertex group $G_v$.
For each vertex $v$, suppose the adjacent edge spaces map onto a collection of disjoint incompressible annuli with union $Q_v$ in $\boundary M_v \setminus P_v$ such that either $(M_v,P_v \cup Q_v)$ is pared or $M_v$ is virtually a solid torus.
Then $G$ is virtually the fundamental group of a hyperbolic $3$--manifold.
\end{thm}

The rough outline of the proof of Theorem~\ref{thm:main} is as follows. 
We are given a group with planar boundary that might not act on $S^2$.  We assume that the boundary is a proper subset of $S^2$ that does not contain an embedded \Sierpinski\ carpet. (If it does contain a \Sierpinski\ carpet, then we must also assume the relative Cannon conjecture.) The main result of \cite{HruskaWalsh} allows us to conclude that each vertex group of its JSJ decomposition acts on $S^2$.  The second step is to apply Theorem~\ref{thm:unpinch}, which allows us to remove any peripheral subgroups that are not virtually abelian. Then we split along an elementary hierarchy---which is known to be finite by \cite{LouderTouikan17}---so that the terminal vertex groups are either virtual surface groups or finite groups (or possibly carpet groups, which are Kleinian by the relative Cannon conjecture).
These vertex groups are covered by Kleinian groups.  We put these vertex groups together using Thurston's hyperbolization theorem. 
At this point we have completed the proof of Theorem~\ref{thm:generalcase}. We conclude that each vertex group of the topmost JSJ splitting is covered by a Kleinian group.  Then we put these together virtually using Theorem~\ref{thm:puttogether}.

This paper is organized as follows. Section~\ref{sec:Convergence} reviews material on convergence groups, relatively hyperbolic group pairs, and their boundaries.  In Section~\ref{sec:Planar} we review some important planar topology which is used in later sections. Section~\ref{sec:unpinch} is devoted to unpinching, and here we prove Theorem~\ref{thm:unpinch}. Section~\ref{sec:Kleinian} is mainly a review of results we will use regarding Kleinian groups, with some slight improvements to the specific cases we will need. 
In Section~\ref{sec:Combination}, we prove three different combination of convergence group theorems for Kleinian actions on $S^2$: for splittings over finite, parabolic, and loxodromic groups. The elementary hierarchy of a relatively hyperbolic group pair and its effect on the boundary is described in Section~\ref{sec:ElementaryHierarchy}, which also  contains the proofs of Theorems~\ref{thm:generalcase}, \ref{thm:puttogether}, and~\ref{thm:main}.
Section~\ref{sec:Applications} contains some applications to various versions of the Cannon conjecture and in this section we prove Theorem~\ref{thm:equivalent}, which has Theorem~\ref{thm:CannonEquivalent} as an immediate corollary.   

\subsection{Acknowledgements}

We thank Daniel Groves, Peter \Haissinsky, and Jason Manning for helpful conversations.  Parts of this paper were developed during visits to Centre International de Rencontres Math\'{e}matiques (CIRM) in Luminy, France and Centre de Recherches Math\'{e}matiques (CRM) in Montreal, Canada. We thank these institutions for their support.  
The first author was partially supported by grant \#714338 from the Simons Foundation, and the second author was supported by NSF-2005353.

\section{Convergence groups and geometrical finiteness}
\label{sec:Convergence}

In this section, we discuss convergence groups and relative hyperbolicity. Since the main focus of this paper is on convergence groups, we define relative hyperbolicity in terms of geometrically finite convergence groups. For several other equivalent characterizations, see \cite{BowditchRelHyp,DrutuSapirTreeGraded,GrovesManning08DehnFilling,Hru-relqconv}.  

An action of a group $G$ on a compact metrizable space $M$ is a \emph{convergence group} action if, for any sequence $(g_i)$ of distinct elements in $G$,
there is a subsequence $(g_{n_i})$ and a pair of points $\zeta,\xi \in M$ such that
\[
   g_{n_i} \big| \bigl( M\setminus\{\zeta\}\bigr) \to \xi
\]
uniformly on compact sets. Such a subsequence is a \emph{collapsing subsequence}, and $\zeta$ and $\xi$ are its \emph{repelling} and \emph{attracting} points.
If $M$ has at least three points, the action is a convergence group action if and only if the action of $G$ on the space of distinct triples of points of $M$ is proper.
A convergence group action on a space $M$ with fewer than three points reduces to any action of a finite group on the empty set, any action of an arbitrary group on a point, and any finitely generated $2$--ended group acting on a two-point space. 
The \emph{limit set} $\Lambda{G}$ of a convergence group is the set of limit points of any orbit. If $M$ has fewer than three points, $M$ is considered to be the limit set by convention. If $\Lambda G$ has fewer than three points then $G$ is \emph{elementary}.
The complement of the limit set is the \emph{ordinary set} or \emph{domain of discontinuity} $\Omega G$.  See \cite{GehringMartin87,Tukia94,Bowditch99ConvergenceGroups} for more background on convergence groups.

A convergence group action of $G$ on $M$ is \emph{geometrically finite} if every point of $\Lambda G$ is either a conical limit point or a bounded parabolic point, which are defined as follows.
A point $\zeta \in M$ is a \emph{conical limit point} if there exists a sequence $(g_i)$ in $G$ and a pair of distinct points $\xi_0\ne\xi_1$ in $M$ such that
\[
   g_i \big| \bigl( M \setminus \{\zeta\}\bigr) \to \xi_0
   \qquad \text{and} \qquad
   g_i(\zeta) \to \xi_1.
\]
A point $\eta\in M$ is \emph{bounded parabolic} if
its stabilizer acts properly and cocompactly on $M \setminus\{\eta\}$.
The stabilizers of the bounded parabolic points are \emph{maximal parabolic subgroups}, and they are always infinite under the conventions above.  If $G$ acts as a geometrically finite convergence group on a compact metrizable space $M$, then $(G, \mathcal{P})$ is a \emph{relatively hyperbolic group pair}, where $\mathcal{P}$ is the collection of maximal parabolic subgroups. In this case, the space $\Lambda(G) \subseteq M$,  is known as the \emph{relatively hyperbolic boundary} or the \emph{Bowditch boundary} and is denoted $\boundary(G,\mathcal{P})$.
A proper action of a group $G$ on a proper $\delta$--hyperbolic space $X$ is \emph{geometrically finite} if the induced action on the Gromov boundary $\partial X$ is a geometrically finite convergence action.

A \emph{geometrically finite Kleinian group action} is a faithful, proper, geometrically finite, isometric action on $\Hyp^3$. Note that it will act as a geometrically finite convergence group on its limit set and also on $S^2$. 
By definition, a geometrically finite Kleinian group is relatively hyperbolic with respect to its maximal parabolic subgroups, and the limit set is its Bowditch boundary. 
In this case, the boundary is planar and the action naturally extends to a geometrically finite convergence action on $\widehat{\C} = \C \cup \{\infty\}$ (see \cite{GehringMartin87}).  In fact, the study of Kleinian groups inspired the study of more general convergence groups (see \cite{Fredenconv,Tukia94,Bowditch99ConvergenceGroups}).  

Geometrical finiteness of a Kleinian group depends only on the topological action by homeomorphisms on $S^2$, a fact that is implicit in \cite{BeardonMaskit74,Maskit88} and is made explicit in \cite{Bowditch95}.  In particular, if two convergence groups on $S^2$ are topologically conjugate and one is geometrically finite, then so is the other.
The topological characterization also implies that any convergence group on $S^2$ with a geometrically finite subgroup of finite index must itself be geometrically finite; see \cite[Prop.~VI.E.6]{Maskit88}.  

\section{Planar topology}
\label{sec:Planar}

This section collects background from various parts of planar topology. For the purposes of this paper, a topological space is \emph{planar} if it embeds in the $2$--sphere $S^2$.

A family of subsets of a metric space is a \emph{null family} if, for each $\epsilon>0$, only finitely many sets of the family have diameter greater than $\epsilon$.
The following is a special case of Moore's classical theorem on decompositions of $S^2$ (see, for instance, \cite[Thm.~61.IV.8]{Kuratowski_VolII} or \cite{Cannon78} for proofs).

\begin{thm}[Moore]
\label{thm:Moore}
Let $\mathcal{D}$ be a null family of pairwise disjoint closed proper subsets of $S^2$.
Suppose each member of $\mathcal{D}$ is connected and does not separate $S^2$.
Then the quotient space $S^2/\mathcal{D}$ is homeomorphic to $S^2$.
\end{thm}

A related theorem is due to Morton Brown (see \cite{Brown60_InverseLimits} or \cite{Ancel87_NearHomeo}). We use this theorem in the ``blow-up'' operation in Section~\ref{sec:unpinch}.  Recall that a continuous map is \emph{monotone} if the pre-image of each point is connected. 

\begin{thm}[Brown]
\label{thm:NearHomeo}
Given an inverse sequence whose factor spaces are all homeomorphic to $S^2$ and whose bonding maps are all monotone surjections, the inverse limit space is also homeomorphic to $S^2$.
\end{thm}

\begin{defn} \label{def:covering}
Let $G,H \le \Homeo(S^2)$ be two convergence groups on $S^2$. Then $H$ is \emph{covered} by $G$ if there is an isomorphism $\phi\colon G \to H$ and a monotone surjective map $\eta\colon S^2\to S^2$ that is $\phi$--equivariant, in the sense that the following diagram commutes: 
\[
\xymatrix{
   S^2 \ar[r]^{g}\ar[d]^{\eta}  &  S^2 \ar[d]^{\eta} \\
   S^2 \ar[r]^{\phi(g)}   &  S^2
}
\]
Furthermore, we require that
$\eta\big| \eta^{-1}\bigl(\Omega(H)\bigr)$
is a homeomorphism.
\end{defn}


The following three geometrization results roughly say that a discrete group of homeomorphisms of a $2$--manifold always preserves a conformal structure and can be realized as a group of spherical, Euclidean, or hyperbolic isometries.
The first of these results, focusing on finite homeomorphism groups, is a classical theorem of \Kerekjarto\ (see \cite{Kolev06} for details).

\begin{thm}[\Kerekjarto]
\label{thm:Kerekjarto}
Every finite group of homeomorphisms of $S^2$ is topologically conjugate to a subgroup of the orthogonal group $O(3)$.
\end{thm}

The following theorem is a consequence of \Kerekjarto's work together with the geometrization of $2$--dimensional orbifolds. See, for instance \cite[Thm.~3.1]{HruskaWalsh} for a proof.

\begin{thm}
\label{thm:OrbifoldGeometrization}
Suppose a group $G$ acts properly by homeomorphisms on a connected surface $X$.
Then $X$ admits a complete metric of constant curvature modeled on either $S^2$, $\E^2$, or $\Hyp^2$ such that the action is isometric.
\end{thm}

If a discrete group action on the open disc extends to a convergence group action on the closed disc, one can identify the disc with the hyperbolic plane by the following theorem due to Martin--Tukia, which strengthens the conclusion of the preceding result.
We note that the much stronger classification theorem for convergence groups acting on $S^1$ of \cite{Tukia_Fuchsian,Gabai92,CassonJungreis94} uses this theorem as a key ingredient.

\begin{thm}[\cite{MartinTukia88}, Thm.~4.4]
\label{thm:MartinTukia}
Let $G \le \Homeo(D^2)$ be any convergence group on the disc $D^2$. Then $G$ is topologically conjugate to a Fuchsian group action on $\Hyp^2 \cup S^1$.
\end{thm}

Recall that a Peano continuum is a compact, connected, locally connected, metrizable space.
The following theorem collects several classical results about the topology of planar Peano continua. See \cite[Thms.~61.II.10, 61.II.4, and 61.I.8$'$]{Kuratowski_VolII} or \cite[Prob.~19-f]{Milnor_Dynamics} for proofs.

\begin{thm}
\label{thm:PlanarPeano}
Let $M \subseteq S^2$ be a nontrivial planar Peano continuum.
   \begin{enumerate}
   \item
   \label{item:PeanoNull}
   The set of components of $S^2 \setminus M$ is a null family.
   \item
   \label{item:PeanoTorhorst}
   The frontier of each component of $S^2 \setminus M$ is itself a Peano continuum.
   \item
   \label{item:PeanoCutPoint}
   For each component $\Delta$ of $S^2 \setminus M$, a cut point of $\boundary\Delta$ is also a cut point of $M$.
   If $M$ has no cut points, then $\boundary\Delta$ is a simple closed curve.
   \end{enumerate}
\end{thm}

In order to better understand the complementary regions of a planar Peano continuum with cut points, we use the classical \Caratheodory--Torhorst theorem on boundary extensions of conformal mappings.
The version stated below combines several closely related results, namely Theorems 17.14, 17.12, and~17.4 of \cite{Milnor_Dynamics} and Proposition~2.5 of \cite{Pommerenke_BoundaryConformal}.

\begin{thm}[\Caratheodory--Torhorst]
\label{thm:Torhorst}
Let $\Delta \subset S^2$ be a simply connected open set whose frontier $\boundary\Delta$ is a nontrivial Peano continuum.
There exists a quotient map $q\colon D^2 \to \bar\Delta$ with the following properties.
\begin{enumerate}
\item
\label{item:TorhorstQuotient}
The map $q$ restricts to a homeomorphism between the open disc and $\Delta$ and restricts to a quotient $q\colon S^1 \to \boundary \Delta$ of frontiers.
\item
\label{item:TorhorstLifting}
Any homeomorphism of the pair $(\bar\Delta,\boundary\Delta)$ lifts to a unique homeomorphism of the pair $(D^2,S^1)$.
\item
\label{item:TorhorstDisconnected}
For each $\zeta \in \boundary\Delta$, the set $q^{-1}(\zeta)\subset S^1$ is totally disconnected.
\item
\label{item:TorhorstCutPoints}
If $\boundary\Delta \setminus \{\zeta\}$ has $m<\infty$ components, then $q^{-1}(\zeta)$ has cardinality $m$.
\end{enumerate}
\end{thm}

Building on the preceding theorem, any convergence group action on $\bar\Delta$ lifts to a convergence group action on $D^2$.  This conclusion can be useful since convergence groups on $D^2$ are always Fuchsian by Theorem~\ref{thm:MartinTukia}.

\begin{prop}[Lifting convergence groups]
\label{prop:LiftingConvergence}
Let $\Delta \subset S^2$ be a simply connected open set with $\boundary\Delta$ a nontrivial Peano continuum.
Let $q \colon D^2 \to \bar\Delta$ be the \Caratheodory--Torhorst quotient map.
Let $G$ act by homeomorphisms on the pair of spaces $(\bar\Delta,\boundary\Delta)$ as a faithful convergence group.
Then the induced action of $G$ on the pair $(D^2,S^1)$ is also a convergence group.
\end{prop}

\begin{proof}
We will see that the convergence property on $D^2$ is a consequence of the convergence property on $\bar\Delta$.  Let $(g_i)$ be a sequence of distinct elements of $G$. After passing to a subsequence, there are points $\zeta,\xi \in \boundary\Delta$ with
\[
   g_i \bigm| \bigl( \bar\Delta \setminus \{\zeta\} \bigr) \to \xi
\]
uniformly on compact sets; \emph{i.e.}, for any compact $K' \subseteq \bar\Delta \setminus \{\zeta\}$ and $L' \subseteq \bar\Delta \setminus \{\xi\}$, the intersection $g_i(K') \cap L'$ is empty for almost all $i$.
It follows that for any compact $K \subset D^2 \setminus q^{-1}(\zeta)$ and $L \subset D^2 \setminus q^{-1}(\xi)$, we have
\begin{equation}
\label{eqn:Collapsing}
\tag{$\star$}
   g_i(K) \cap L = \emptyset
   \quad \text{for almost all $i$}.
\end{equation}

Choose an exhaustion $K_1 \subset K_2 \subset \cdots$ of $X_\zeta = D^2 \setminus q^{-1}(\zeta)$ by compact connected sets whose complementary components each have noncompact closure.
Then $D^2 \setminus K_m$ has only finitely many components, each of which intersects $q^{-1}(\zeta)$. Choose a similar exhaustion $L_1 \subset L_2 \subset \cdots$ of $X_\xi = D^2 \setminus q^{-1}(\xi)$.
By \eqref{eqn:Collapsing}, for each $j$ the compact set $g_i^{-1}(L_j)$ is disjoint from $K_j$ for almost all $i$. By a diagonal argument, there is a subsequence of $(g_i)$ so that for each $j$ the connected set $g_i^{-1}(L_j)$ lies in a fixed open component $U_j \subseteq U_{j-1}$ of $D^2 \setminus K_j$ whenever $i \ge j$.  Since $L_j \subseteq g_i(U_j)$, we can pass to a further subsequence so that for each $j$ the connected set $g_i(D^2 \setminus U_j)$ lies in a fixed open component $V_j \subseteq V_{j-1}$ of $D^2 \setminus L_j$ whenever $i \ge j$. 

By Theorem~\ref{thm:Torhorst}(\ref{item:TorhorstDisconnected}), any two points of $q^{-1}(\zeta)$ are separated by $K_j$ for some $j$, so the intersection $\bigcap_j U_j$ is a single point $\hat{\zeta} \in q^{-1}(\zeta)$ and $\{U_m\}$ is a neighborhood base at $\hat{\zeta}$.
Similarly, $\bigcap_j V_j$ is a single point $\hat{\xi} \in q^{-1}(\xi)$, and $\{V_n\}$ is a neighborhood base at $\hat{\xi}$.
For each $m,n$ choose $j \ge m,n$. Then
\[
   g_i(D^2 \setminus U_m) \subseteq g_i(D^2 \setminus U_j) \subseteq V_j \subseteq V_n
\]
holds when $i \ge j$, so $g_i \big| \bigl(D^2 \setminus \{\hat{\zeta}\} \bigr) \to \hat{\xi}$ uniformly on compact sets.
\end{proof}

Noting the similarity between \eqref{eqn:Collapsing} and a proper action, we recover the following folk result, which by \cite[\S 2]{Tukia94} implies that the number of ends of a finitely generated group is $0,1,2,$ or $\infty$. 

\begin{cor}
If $X$ is any connected, locally connected, locally compact metrizable space and $G$ acts properly on $X$, the extension to the Freudenthal compactification $X \cup \Ends(X)$ is a convergence group action of $G$.

In particular, if $G$ is finitely generated, the action of $G$ on $\Ends(G)$ is a convergence group action. \qed
\end{cor}



\section{Unpinching}
\label{sec:unpinch}

The goal of this section is to prove the following theorem, which states that any geometrically finite convergence group on $S^2$ is covered by such an action in which the parabolic subgroups are all virtually abelian. 

\begin{thm}
\label{thm:CoveringConvergence}
Suppose $(G,\mathcal{P})$ is relatively hyperbolic, and we have a partition $\mathcal{P}=\mathcal{Q} \sqcup \mathcal{P}_h$ into conjugacy invariant subfamilies such that all members of $\mathcal{P}_h$ are word hyperbolic.
Suppose the boundary $M=\boundary(G,\mathcal{P})$ is a subspace of $S^2$ and the action on $M$ extends to a convergence group action on $S^2$.

Then the boundary $\check{M} = \partial(G, \mathcal{Q})$ is planar and the action on $\check{M}$ extends to a convergence group action on $S^2$ covering the action of $(G, \mathcal{P})$ on $S^2$.  
\end{thm} 

The proof of this theorem relies on the following result, which combines results of \Drutu--Sapir \cite[Cor.~1.14]{DrutuSapirTreeGraded} and Wen-yuan Yang \cite[Lem.~4.14]{Yang14_Peripheral}.

\begin{thm}[\Drutu--Sapir, Yang]
\label{thm:CollapsingBoundaries}
Suppose $(G,\mathcal{P})$ is relatively hyperbolic, and we have a partition $\mathcal{P}=\mathcal{Q} \sqcup \mathcal{P}_h$ into conjugacy invariant subfamilies such that all members of $\mathcal{P}_h$ are word hyperbolic.
Then $(G,\mathcal{Q})$ is also relatively hyperbolic, and there is a $G$--equivariant quotient map $\boundary(G,\mathcal{Q}) \to \boundary(G,\mathcal{P})$ obtained by collapsing the limit set of each member of $\mathcal{P}_h$ to a point.
\end{thm}

Let $\mathbb{Q}$ and $\mathbb{P}_h$ be representatives of the finitely many conjugacy classes of members of $\mathcal{Q}$ and $\mathcal{P}_h$.
There is a natural one-to-one correspondence between cosets $gQ$ with $g\in G$ and $Q\in \mathbb{Q}$ and conjugates $gQg^{-1} \in \mathcal{Q}$, and a similar one-to-one correspondence for cosets of members of $\mathbb{P}_h$.
(This correspondence is bijective since each member of $\mathcal{P}$ is equal to its own normalizer.)

\begin{lem}
\label{lem:PartiallyCusped}
Choose $P \in \mathbb{P}_h$. Let $Y_P$ be the space obtained from the Cayley graph of $G$ by attaching a combinatorial horoball to each coset $gH$ with $g \in G$ and $H \in \mathbb{Q}\cup \mathbb{P}_h$ except for the trivial coset $P$.
Then $Y_P$ is $\delta$--hyperbolic, and its Gromov boundary is homeomorphic to the quotient space $M_P$ formed from $\check{M}= \boundary(G,\mathcal{Q})$ by collapsing the limit set of each $gH$ with $H \in \mathbb{P}_h$ to a point except for the trivial coset $P$.
\end{lem}

This lemma generalizes Theorem~\ref{thm:CollapsingBoundaries}, but here the operations disregard the action of $G$.
The lemma relies on a result of Sisto \cite[Prop.~4.9]{Sisto_MetricRelHyp}, which is a geometric generalization of a theorem of Groves--Manning \cite{GrovesManning08DehnFilling}.

\begin{proof}  
In \cite{DrutuSapirTreeGraded}, \Drutu--Sapir study relatively hyperbolic spaces and groups in terms of the structure of their asymptotic cones.
In particular, they show that any asymptotic cone of such a space is {\emph tree graded} with respect to certain subspaces called pieces, which means that every simple closed curve is contained in a piece.

If a space $K$ is tree graded with respect to a collection of pieces $\mathcal{H}$, and some piece $H$ is a tree, then $K$ is tree graded with respect to the smaller collection of pieces $\mathcal{H} \setminus \{H\}$.  This simple observation is a key idea in the first part of our proof. 

We claim that if a geodesic space $X$ is asymptotically tree graded with respect to a collection $\mathcal{A}$ of subspaces and some $A \in \mathcal{A}$ is $\delta$--hyperbolic, then $X$ is asymptotically tree graded with respect to $\mathcal{A} \setminus \{A\}$.
To show the claim, we first note that $A$ is quasiconvex in $X$ (see \cite[Lem.~4.3]{DrutuSapirTreeGraded}), so
$A$ with the subspace metric is a $\delta$--hyperbolic $(1,\epsilon)$--quasigeodesic space, and any asymptotic cone of $A$ is a $0$--hyperbolic geodesic space---\emph{i.e.}, an $\R$--tree \cite[\S 3.4]{CDP90}.
Thus, $\text{Cone}(X)$ is also tree graded with respect to the ultralimits of members of $\mathcal{A} \setminus \{A\}$, which establishes the claim.

Since $(G, \mathcal{Q})$ is a relatively hyperbolic group pair, the cusped space $Y$ formed by gluing combinatorial horoballs to $G$ along all left cosets of members of $\mathbb{Q}$ is hyperbolic, as is the cusped space $Z$ formed by gluing combinatorial horoballs to $G$ along all left cosets of members of $\mathbb{P} = \mathbb{Q} \cup \mathbb{P}_h$, by \cite{GrovesManning08DehnFilling}.  
Observe that $Z$ is also obtained from $Y$ by gluing horoballs to the left cosets of members of $\mathbb{P}_h$.  Since $Z$ is hyperbolic, \cite[Prop.~4.9]{Sisto_MetricRelHyp} implies that the cusped space $Y$ is asymptotically tree graded with respect to the set of left cosets of members of $\mathbb{P}_h$.  If $P$ is a single member of $\mathbb{P}_h$, then $Y$ is also asymptotically tree graded with respect to the family $\set{gP'}{\text{$g\in G$ and $P' \in \mathbb{P}_h$}} \setminus \{P\}$ by the claim above.
Therefore, if horoballs are attached to ${Y}$ along each left coset in the family $\set{gP'}{\text{$g\in G$ and $P' \in \mathbb{P}_h$}} \setminus \{P\}$, we get a new hyperbolic space $Y_P$.

To complete the proof, we show that the Gromov boundary of $Y_P$ can be obtained as a certain quotient space of the Gromov boundary of $Y$.
Consider the inclusion $Y \to Y_P$ of $\delta$--hyperbolic spaces.  A geodesic ray in $Y$ has an image in $Y_P$ that is typically not a geodesic.
The precise relation between the geodesics of $Y$ and the geodesics of $Y_P$ is described in Sisto \cite[Prop.~4.9]{Sisto_MetricRelHyp}.
The description in \cite[Prop.~4.9]{Sisto_MetricRelHyp} shows that the inclusion $Y \to Y_P$ induces a quotient map $\boundary Y \to \boundary Y_P$ that collapses the limit set of each peripheral coset---except for the trivial coset $P$---to a point in the new boundary.
\end{proof}

In Lemma~\ref{lem:CollapsedBoundary}, we will show that the ``larger'' Bowditch boundary $\check{M} = \boundary(G,\mathcal{Q})$ can be described as an inverse limit of the spaces $M_P$ arising in the previous lemma. 

Before stating the lemma, we introduce some terminology.
Suppose $(G,\mathcal{P})$ and $\mathcal{P} = \mathcal{Q} \sqcup \mathcal{P}_h$ are as in the statement of Theorem~\ref{thm:CoveringConvergence}.
For each $P \in \mathcal{P}_h$, let $\omega_P\colon \check{M} \to M_P$ be the quotient obtained by collapsing the limit set of each member of $\mathcal{P}_h \setminus \{P\}$ to a point (as above), and let $\pi_P \colon M_P \to M$ be the quotient collapsing the limit set of $P$ to a point.
By Lemma~\ref{lem:PartiallyCusped}, each $M_P$ is Hausdorff. The system of spaces $\set{M_P}{P\in \mathcal{P}_h} \cup \{M\}$ and maps $\pi_P$ forms an inverse system (although its index set is not directed). 
We let $\varprojlim M_P$ denote the corresponding inverse limit, which is by definition the pullback of the family of maps $\pi_P$.

We note that, although $G$ does not act on $M_P$, the subgroup $P$ does, since $\check{M} \to M_P$ is $P$--equivariant.

\begin{lem}
\label{lem:CollapsedBoundary}
The inverse limit $\varprojlim M_{P}$ is homeomorphic to $\check{M}= \boundary(G,\mathcal{Q})$. 
\end{lem} 

\begin{proof}
All spaces involved are compact and Hausdorff, so the maps $\omega_P$ induce a continuous closed map $\omega\colon \check{M} \to \varprojlim M_P$.
If $\xi\in M$ is not fixed by $P$ then $\pi_P^{-1}(\xi)$ is a single point of $M_P$. If $\xi$ is fixed by $P$, then $\pi_P^{-1}(\xi) = \Lambda P \subseteq M_P$. 
Thus, the preimage of $\xi$ in $\varprojlim M_P$ is naturally homeomorphic to $\Lambda P$ if $\xi$ is fixed by some $P \in \mathcal{P}_h$ and is a single point if $\xi$ is fixed by no member of $\mathcal{P}_h$.
In particular, $\omega$ naturally induces a bijection between the preimages of $\xi$ in $\check{M}$ and in $\varprojlim M_P$. 
Therefore, each point of $\varprojlim M_P$ has exactly one preimage in $\check{M}$ under the map $\omega$; in other words, $\omega$ is a bijection.
\end{proof}

\begin{lem}[Blowing up one peripheral subgroup]
\label{lem:BlowUpOnce}
Suppose, as in Theorem~\ref{thm:CoveringConvergence}, that the boundary $M=\boundary(G,\mathcal{P})$ is a subspace of $S^2$ and the action on $M$ extends to a convergence group action on $S^2$.
Choose any $P \in \mathcal{P}_h$ with parabolic fixed point $\xi \in M$, and let $M_P$ be as above. Then there exists a Fuchsian action of $P$ on a $2$--sphere $S^2_P$ with a $P$--invariant disc $D^2_P$ and a $P$--equivariant embedding $i_P\colon M_P \to S^2_P$  such that the image of $M_P$ is disjoint from the interior of $D^2_P$.
If $(M_P,\Lambda P) \to (M,\xi)$ and $(S^2_P,D^2_P)\to (S^2,\xi)$ are the natural quotient maps, then the following diagram commutes and all maps are $P$--equivariant:
\begin{equation}
\label{eqn:BlowUpOnce}
\tag{$\dag$}
\begin{gathered}
\xymatrix{
   (M_P,\Lambda P) \ar[r]\ar[d]^{i_P}  &  (M,\xi) \ar[d]^i \\
   (S^2_P,D^2_P) \ar[r]   &  (S^2,\xi)
}
\end{gathered}
\end{equation}
\end{lem}

\begin{proof}
Let $Y_P$ be the space obtained from the Cayley graph of $G$ by attaching a combinatorial horoball to each subgraph stabilized by an element of $\mathcal{P} \setminus \{P\}$.
Then the Gromov boundary of $Y_P$ is equal to the quotient space $M_P$
by Lemma~\ref{lem:PartiallyCusped}.
Let $K_P$ be the quasiconvex subgraph of $Y_P$ corresponding to $P$.

Note that $M$ is the quotient of $M_P$ formed by collapsing $\Lambda P$ to a point $\{\xi\}$.
Let $\mathring{M}_P$ denote the space $M_P \setminus \Lambda P = M \setminus \{\xi\}$.
By hypothesis, $\mathring{M}_P$ embeds in the plane $\R^2 = S^2 \setminus \{\xi\}$, and the action of $P$ on $\mathring{M}_P$ extends to a proper action by homeomorphisms on $\R^2 = S^2 \setminus \{\xi\}$. 
By Theorem~\ref{thm:OrbifoldGeometrization}, there exists a $P$--invariant metric of constant curvature on $\R^2$ on which $P$ acts isometrically. 
We may assume that this metric is hyperbolic, since either $P$ is virtually cyclic or contains a nonabelian free group. 
The space $\mathring{M}_P$ embeds in this hyperbolic plane $\mathcal{H}$ such that the orbit $Px$ of some point $x$ in $\mathring{M}_P$ is a quasiconvex subspace of $\mathcal{H}$.  The limit set $\Lambda P$
in $M_P$ is the set of accumulation points of $K_P$ in $\boundary Y_P$, which may be identified with the Gromov boundary $\boundary P$ of the hyperbolic group $P$. Thus, it coincides with the limit set of the orbit $Px$ in $\mathcal{H}$.
It follows that the embedding $\mathring{M}_P \to \mathcal{H}$ extends to an embedding of $M_P = \mathring{M}_P \cup \Lambda P$ into the closed disk $\bar{\mathcal{H}} = D^2$. 
\end{proof}

Theorem~\ref{thm:CoveringConvergence} has some similarities with an unpublished result of de Souza \cite{deSouza_Blowups}, which has different hypotheses. Parts of the proof below are inspired by de Souza's methods of proof.

\begin{proof}[Proof of Theorem~\ref{thm:CoveringConvergence}]
Throughout the proof, we assume without loss of generality that the subgroups of $\mathcal{P}_h$ lie in one conjugacy class with representative $P$. Let $\xi\in M= \boundary(G,\mathcal{Q} \cup \mathcal{P}_h)$ be the parabolic point fixed by $P$.

\textbf{Claim:} The boundary $\check{M}= \boundary(G,\mathcal{Q})$ is planar.

Recall that $\check{M}$ is the inverse limit $\varprojlim M_H$ over the family of all $H \in \mathcal{P}_h$ by Lemma~\ref{lem:CollapsedBoundary}.
By hypothesis $M= \partial(G,\mathcal{P})$ embeds in a $2$--sphere, which we denote by $S^2$. Let $i_P \colon (M_P,\Lambda P) \to (S^2_P,D^2_P)$ be the $P$--equivariant map of pairs given by Lemma~\ref{lem:BlowUpOnce}.

Fix a set $A\subset G$ of representatives for the cosets $gP$ in $G$ such that the trivial coset $P$ is represented by the identity. For each conjugate $H=aPa^{-1}$ with $a \in A$, we let $S^2_H$ be a copy of $S^2_P$ on which $H$ acts by the rule $h(x) = (a^{-1}ha)(x)$. We similarly use $a$ to identify $M_H$ with a copy of $M_P$ (see Figure~\ref{fig:inverse}). 
\begin{figure}
\[
\xymatrix{
   (M_P,\Lambda P) \ar[r]^{a}\ar[d]  &  (M_H,\Lambda H) \ar[d] \\
   (M, \xi) \ar[r]^{a} &  \bigl(M,a(\xi)\bigr)
}
\qquad\qquad\qquad
\xymatrix{
   (S^2_P,D^2_P) \ar[r]^{a}\ar[d]  &  (S^2_H,D^2_H) \ar[d] \\
   (S^2, \xi) \ar[r]^{a} & \bigl( S^2,a(\xi) \bigr)
}
\]
\caption{If $H=aPa^{-1}$, we use $a$ to identify $M_H$ with $M_P$ and $S^2_H$ with $S^2_P$ as shown.
Each space in the left-hand diagram embeds in the respective space of the right-hand diagram. The embeddings $M_H \to S^2_H$ induce a limit embedding $\check{M} \to \check{S}^2$.}
\label{fig:inverse}
\end{figure} 
Then there is a natural $H$--equivariant embedding $i_H \colon (M_H,\Lambda H) \to (S^2_H,D^2_H)$, which satisfies a commutative diagram analogous to \eqref{eqn:BlowUpOnce} in Lemma~\ref{lem:BlowUpOnce}.

The system of $2$--spheres $\set{S^2_H}{H=aPa^{-1} \in \mathcal{P}_h} \cup \{S^2\}$ and maps $S^2_H \to S^2$ collapsing $D^2_H$ to the parabolic point $a(\xi)$ is an inverse system (whose index set is not directed). Let $\check{S}^2 = \varprojlim S^2_H$ denote its inverse limit, which is by definition the pullback of the family of maps $S^2_H \to S^2$.
If we enumerate the parabolic points, this inverse limit is equivalent to the limit of an inverse sequence (indexed by $\mathbb{N}$) of $2$--spheres in which the $j$th bonding map blows up the $j$th parabolic point to a disc. Thus, $\check{S}^2$ is a $2$--sphere by a theorem of M.~Brown (Theorem~\ref{thm:NearHomeo}).
The embeddings $i_H$ induce an embedding of inverse limits $i \colon \check{M} \to \check{S}^2$. In particular, $\check{M}$ is planar.

\textbf{Claim:} The given $G$--action on $\check{M}$ naturally extends to a $G$--action on $\check{S}^2$.

Suppose $H \in \mathcal{P}_h$ with $H=aPa^{-1}$ for $a \in A$. If $g \in G$ then $ga\in bP$ for some $b \in A$.
The element $g\in G$ induces a natural homeomorphism $S^2_H \to S^2_{bPb^{-1}}$ with rule $g(x) = b^{-1}ga(x)$---using the definition of $S^2_H$ and $S^2_{bPb^{-1}}$ as copies of $S^2_P$ and the fact that $b^{-1}ga \in P$.

The homeomorphisms defined above determine an action of $G$ on the product $\prod_{H \in \mathcal{P}_h} S^2_H$ by permuting the factors, which induces an action on the inverse limit $\check{S}^2 = \varprojlim S^2_H$. 
This action on $\check{S}^2$ extends the given action of $G$ on $\check{M}$.
Let $\pi \colon \check{M}\to M$  and $\pi \colon \check{S}^2 \to S^2$ be the natural projections of inverse limits; \emph{i.e.} collapsing each limit set $\Lambda H$ and each disc $D^2_H$ to a point.
We have a commutative diagram of $G$--equivariant maps
\[
\xymatrix{
   \check{M} \ar[r]^{i} \ar[d]^{\pi}  &  \check{S}^2 \ar[d]^{\pi} \\
   M \ar[r]^{i}   &  S^2
}
\]
For each $H\in \mathcal{P}_h$ the vertical projections above factor $H$--equivariantly as
\[
   \check{M} \longrightarrow M_H \longrightarrow M
   \qquad \text{and} \qquad
   \check{S}^2 \longrightarrow S^2_H \longrightarrow S^2.
\]
By Lemma~\ref{lem:BlowUpOnce}, the action of $H$ on $S^2_H$ is Fuchsian. In particular, $H$ acts as a convergence group on both $M_H$ and $S^2_H$.

\textbf{Claim:} The action of $H$ on $\check{S}^2$ is a convergence group action.

Let $(h_i)$ be a sequence of distinct group elements of $H$. The action on $S^2_H$ has a collapsing subsequence $(h_{n_i})$ with attracting and repelling points $\xi,\zeta\in S^2_H$ respectively. Since $\xi,\zeta \in D^2_H$, each has a unique preimage in $\check{S}^2$. Thus, $(h_{n_i})$ is also a collapsing sequence for the action of $H$ on $\check{S}^2$.

\textbf{Claim:} The action of $G$ on $\check{S}^2$ is a convergence group action.

We will show that $G$ acts properly on the space of distinct triples $\Theta(\check{S}^2)$.
Endow $\check{S}^2$ with any metric compatible with its topology.
Choose $\epsilon>0$ and let $K_\epsilon\subset \Theta(\check{S}^2)$ be the compact set of all triples whose pairwise distances are at least $\epsilon$.
Since $\Theta(\check{S}^2)$ is exhausted by the compact sets $K_\epsilon$, it suffices to show that each $K_\epsilon$ meets only finitely many of its $G$--translates.

Assume, by way of contradiction, that infinitely many $G$--translates $g_i (K_\epsilon)$ intersect $K_\epsilon$ for some $\epsilon>0$.  
If any coset $Hg$ with $H \in \mathcal{P}_h$ and $g\in G$ contains infinitely many elements of the sequence $(g_i)$, then $g_i g^{-1} \in H$ for infinitely many $i$.
In this case, $K_\epsilon \cup gK_\epsilon$ intersects infinitely many of its $H$--translates, contradicting the fact that $H$ is a convergence group on $\check{S}^2$.
Thus, we may assume that for each $H\in \mathcal{P}_h$ and each $g \in G$, only finitely many elements of the sequence $(g_i)$ lie in the coset $Hg$.

Since $K_\epsilon$ intersects $g_i(K_\epsilon)$, there is a triple $p_i \in K_\epsilon$ with $g_i(p_i) \in K_\epsilon$ for each $i$.   
Although $\pi\colon \check{S}^2 \to S^2$ is not injective, it induces a map $\pi \colon \Theta(\check{S}^2) \to \Theta(S^2) \cup \Delta$, where $\Delta$ is the diagonal of the triple space.
Since $G$ acts properly on $\Theta(S^2)$, the sequences $\pi(p_i)$ and $\pi\bigl( g_i(p_i) \bigr)=g_i\bigl(\pi(p_i) \bigr)$ cannot both lie in a compact subspace of $\Theta(S^2)$.
Thus, at least one of $\bigl( \pi(p_i)\bigr)$ or $\bigl(\pi (g_i(p_i))\bigr)$ escapes every compact set in $\Theta(S^2)$.
Note that the second case reduces to the first if one replaces $p_i$ with $g_i(p_i)$ and considers $(g_i^{-1})$ in place of $(g_i)$.

Suppose $\bigl( \pi(p_i) \bigr)$ escapes from every compact set in $\Theta(S^2)$. If $\pi(p_i) \in \Delta$ for infinitely many $i$, then each such triple $p_i$ contains two points in the same disc $D^2_{H_i}$ of $\check{S}^2$.
The diameter of $D^2_{H_i}$ is at least $\epsilon$ since $p_i \in K_\epsilon$, so some disc $D^2_H$ occurs infinitely often, which implies that infinitely many $g_i$ are in the same coset of $H$, a contradiction.
Thus, we may assume $\pi(p_i)\in \Theta(S^2)$; in other words, $\pi(p_i)$ is a distinct triple for each $i$.

Since the action of $G$ on $S^2$ is a convergence group, we can replace $(g_i)$ with a collapsing subsequence for the action on $S^2$ with attracting and repelling points $\xi,\zeta \in S^2$ respectively.
Each triple $\pi(p_i)$ contains at least two points $x'_i,y'_i$ not equal to $\zeta$. Thus, $g_i(x'_i)$ and $g_i(y'_i)$ both converge to $\xi$. Since $g_i(p_i) \in K_\epsilon$, the sequences $g_i(x_i)$ and $g_i(y_i)$ in $\check{S}^2$ cannot converge to the same point. So $\pi^{-1}(\xi)$ must be a disk $D^2_H$ for some $H \in \mathcal{P}_h$.

Pass to a subsequence so that the terms of $(g_i)$ lie in distinct $H$--cosets. 
Recall that $M=\boundary(G,\mathcal{Q} \cup \mathcal{P}_h)$ is the Gromov boundary of a $\delta$--hyperbolic cusped space $X$ on which $G$ acts properly \cite{BowditchRelHyp}. The action of $G$ on $X \cup M$ is a convergence group action \cite[Thm.~3A]{Tukia94}. Choose a basepoint $o \in X$ with orbit $G(o)$. The action of $H$ on $G(o) \cup \bigl( M \setminus \{\xi\} \bigr)$ has a compact fundamental domain $F$ (see \cite[Prop.~6.5]{BowditchRelHyp}). 
For each $i$, choose $h_i \in H$ with $h_i^{-1} g_i(o)\in F$, and let $f_i = h_i^{-1} g_i$, so that $g_i = h_i f_i$. Note that $(f_i)$ is a sequence of distinct elements of $G$, since the $g_i$ lie in distinct $H$--cosets.

For each fixed $h \in H$, the sequence $h f_i(o)$ lies in the compact set $hF$ for all $i$, so $\bigl( h f_i(o) \bigr)$ does not have $\xi$ as an accumulation point. Since $o\ne \zeta$, we have $g_i(o) = h_i f_i(o) \to \xi$ as $i \to \infty$. Thus, the sequence $(h_i)$ repeats each term only finitely many times. Pass to a subsequence so that $(h_i)$ is a sequence of distinct elements of $H$.

Replace $(g_i) = (h_i f_i)$ with a further subsequence such that $(h_i)$ is a collapsing sequence for the $H$--action on $\check{S}^2$ with attracting and repelling points $\mu,\nu \in \boundary D^2_H = \pi^{-1}(\xi)$ respectively, and such that $(f_i)$ is a collapsing sequence for the action on $M$ and also for the extended convergence actions on $S^2$ and on $X \cup M$. Let $\alpha,\beta\in M$ be the attracting and repelling points of $(f_i)$. Since $f_i(o) \in F$ and $o \ne \beta$, it follows that $f_i(o) \to \alpha \in F$. In particular, $\alpha \ne \xi$.

Choose a compact set $K \subset \check{S}^2 \setminus \pi^{-1}(\beta)$. Then $f_i(\pi K)$ converges to the point $\alpha$ in $S^2$. Since $\pi(\nu)=\xi\ne \alpha$, we have
$\nu \notin \pi^{-1}(\alpha)$. So the sets $g_i(K) = h_i f_i(K)$ converge to the attracting point $\mu$ of $(h_i)$. In other words, the restriction of $g_i$ to $\check{S}^2 \setminus \pi^{-1}(\beta)$ converges uniformly on compact sets to a constant.
This is a contradiction, since the points of the triple $p_i$ are separated by a distance at least $\epsilon$, and at least two of them are not in $\pi^{-1}(\beta)$. Therefore, the action of $G$ on $\Theta(\check{S}^2)$ must be proper; \emph{i.e.}, the action of $G$ on $\check{S}^2$ is a convergence group action.

\textbf{Claim:}
The action of $G$ on $\check{S}^2$ covers the action of $G$ on $S^2$.

The quotient $\pi\colon \check{S}^2 \to S^2$ collapses each disc $D_g$ to a point, so it is monotone.
By definition, $\pi$ is $G$--equivariant.
Also observe that $\Omega(G) \subset S^2$ is disjoint from the images of the discs $D_g$, so the injective quotient map $\pi | \pi^{-1}(\Omega G)$ is a homeomorphism.
\end{proof}

\section{Kleinian groups and $3$--manifolds}
\label{sec:Kleinian}

In this section, we discuss some theorems and their consequences regarding hyperbolic structures on $3$--manifolds and their induced Kleinian actions on $\widehat{\C}$. Many theorems in this section are variations of well-known results that will be used throughout the rest of this paper.

\begin{defn}[Pared $3$--manifolds]
\label{defn:Pared}
A \emph{pared $3$--manifold} is a pair $(M,P)$ satisfying the following conditions:
   \begin{enumerate}
   \item $M$ is a compact orientable irreducible $3$--manifold.
   \item $P \subseteq \boundary M$ is a disjoint union of incompressible tori and annuli.
   \item No two components of $P$ are isotopic in $\boundary M$.
   \item Every noncyclic abelian subgroup of $\pi_1(M)$ is conjugate to a subgroup of $\pi_1(P_0)$ for some component $P_0$ of $P$.
   \item There are no essential cylinders $(A,\boundary A) \hookrightarrow (M,P)$.
   \end{enumerate}
The submanifold $P$ is the \emph{parabolic locus of $M$}.
\end{defn}

Similarly, a \emph{pared $3$--orbifold} is the quotient of a pared $3$--manifold by a finite group of homeomorphisms that leaves the pared structure invariant.  

The Hyperbolization Theorem has the following form for pared $3$--manifolds with boundary (see \cite{Kapovich01}).

\begin{thm}[Thurston's Hyperbolization]
\label{thm:Hyperbolization}
If $(M,P)$ is a pared $3$--manifold with $\boundary M$ nonempty, then $M$ admits a complete geometrically finite hyperbolic structure with parabolic locus $P$.
\end{thm}

By work of Bowditch, geometrical finiteness implies relative hyperbolicity in the following sense (see \cite{Bowditch93,BowditchRelHyp}).

\begin{cor}
Let $(M,P)$ be a pared $3$--manifold with nonempty boundary.
Then $G=\pi_1(M)$ is relatively hyperbolic with respect to the fundamental groups of components of the parabolic locus.
\end{cor}

In the next result, we examine assembling several pared $3$--manifolds by gluing them along subsurfaces of their boundaries to form a graph of spaces in the sense of Scott--Wall \cite{ScottWall79}.
The idea of pared manifolds has been used extensively in Kleinian groups, and although we did not find the following in the literature, we expect that it was known previously. 

\begin{prop}[Pared Combination]
\label{prop:ParedCombination}
Let $M$ be the total space of a graph of spaces  such that each vertex space is a pared $3$--manifold $(M_v,P_v)$ and each edge space is an annulus.  For each vertex $v$, suppose the adjacent edge spaces map homeomorphically onto a collection of disjoint incompressible annuli with union $Q_v$ in $\boundary M_v \setminus P_v$ such that $(M_v,P_v \cup Q_v)$ is also pared.
Then $(M,P)$ is again a compact atoroidal pared $3$--manifold, where $P = \bigcup_v P_v$.
\end{prop}

\begin{proof}
By hypothesis, the submanifold $Q_v$ of $M_v$ is a disjoint union of finitely many incompressible annuli in $\boundary M_v$, and the inclusion of each component of $Q_v$ into $M_v$ is $\pi_1$--injective.
Since $\pi_1(M)$ is assembled as the fundamental group of a graph of groups, each inclusion $P_v \to M_v \to M$ is $\pi_1$--injective.  Thus, $P$ is incompressible in $M$.
Since $M$ is formed by gluing finitely many compact irreducible $3$--manifolds $M_v$ along incompressible subsurfaces of their boundaries, $M$ is itself irreducible.

Since each group $\pi_1(G_v)$ is hyperbolic relative to the fundamental groups of the components of $P_v$, Dahmani's combination theorem \cite{Dahmani03Combination} gives that $G$ is hyperbolic relative to the fundamental groups of components of $P$.
In a relatively hyperbolic group, any subgroup that is not virtually cyclic and has no nonabelian free subgroup must be parabolic (\cite[Thm.~2U]{Tukia94}). Thus, every $\Z^2$ subgroup of $G$ is conjugate to a subgroup of the fundamental group of a component of $P$. Furthermore, the fundamental groups of components of $P$ are a malnormal family in $\pi_1(M)$. Thus, no two components of $P$ are isotopic, and $(M,P)$ has no essential cylinder $(A,\boundary A) \hookrightarrow (M,P)$.
\end{proof}

We occasionally use Marden's Isomorphism Theorem when studying topological conjugation. The version stated here is part of \cite[Thm.~3.8]{Tukia85}. (Note that Mostow--Prasad rigidity is the special case when $\Omega(\Gamma)$ is empty.)

\begin{thm}[Marden's Isomorphism Theorem]
\label{thm:MardenIsomorphism}
Let $\Gamma$ and $\Gamma'$ be geometrically finite Kleinian groups, and let $\phi\colon \Gamma\to \Gamma'$ be an isomorphism of groups.
Any $\phi$--equivariant homeomorphism $f\colon \Omega(\Gamma) \to \Omega(\Gamma')$ extends uniquely to a $\phi$--equivariant homeomorphism $\bar{f}\colon \widehat{\C} \to \widehat{\C}$. Furthermore:
   \begin{enumerate}
   \item
   \label{item:MardenIdentity}
   If $\phi\colon \Gamma \to \Gamma$ is the identity, then $\bar{f}\big| \Lambda(\Gamma)$ equals the identity.
   \item
   \label{item:MardenConformal}
   If $f$ is conformal then $\bar{f}$ is also conformal.
   \end{enumerate}
\end{thm}

\begin{cor}
\label{cor:GeomFiniteManifolds}
Let $(M,P)$ be a pared $3$--manifold with two complete geometrically finite hyperbolic structures.
Let $\Gamma$ and $\Gamma'$ be the corresponding Kleinian representations of $\pi_1(M)$.
Then the actions of $\Gamma$ and $\Gamma'$ on $\widehat{\C}$ are topologically conjugate.
\end{cor}

\begin{proof}
The $3$--manifolds $\Hyp^3 \cup \Omega(\Gamma) / \Gamma$ and $\Hyp^3 \cup \Omega(\Gamma') / \Gamma'$ are homeomorphic.  Let $\phi\colon \Gamma\to\Gamma'$ be the isomorphism induced by such a homeomorphism.
Lifting to the universal cover gives a $\phi$--equivariant homeomorphism $\Omega(\Gamma) \to \Omega(\Gamma')$. 
The result now follows by Theorem~\ref{thm:MardenIsomorphism}.
\end{proof}

Mapping classes of finite type surfaces naturally act on the circle at infinity of the hyperbolic plane. The compact case of this result is well-known; see \cite{Calegari_Foliations} or \cite{FarbMargalit_Primer}. 
For the general case, see Cantwell--Conlon \cite{CantwellConlon15}.
(The first two assertions also follow from the more general Theorem~\ref{thm:MardenIsomorphism}.)

\begin{thm}
\label{thm:CircleAction}
Let $X$ be a finite area, complete, connected hyperbolic surface.
   \begin{enumerate}
   \item If $f$ is a homeomorphism of $X$, then any lift $\tilde{f}$ to a homeomorphism of\/ $\Hyp^2$ extends uniquely to a homeomorphism $\hat{f}$ of $D^2=\Hyp^2 \cup S^1$.
   \item Suppose homeomorphisms $f$ and $g$ of $X$ are homotopic.  For each lift $\tilde{f}$ of $f$, let $\tilde{g}$ be the lift of $g$ homotopic to $\tilde{f}$ by a lift of the given homotopy. Then the restrictions of $\hat{f}$ and $\hat{g}$ to $S^1$ agree.
   \item If $\hat{f}$ equals the identity on $S^1$, then $f$ is isotopic to the identity.
   \end{enumerate}
\end{thm}

If a finite group $G$ of mapping classes of a finite type surface fixes a point of Teichm\"{u}ller space, then $G$ is isomorphic to a group of isometries of the corresponding hyperbolic metric (see, for instance, \cite[\S 12.1]{FarbMargalit_Primer}).
The following variation of this fact deals with a finite group of homeomorphisms, rather than mapping classes.

\begin{prop}
\label{prop:Isotopic}
Let $X$ be a finite area complete, connected hyperbolic surface. Let $G$ be a finite group of homeomorphisms of $X$ such that each element is homotopic to an isometry of $X$.
Then $G$ is topologically conjugate to a group of isometries of $X$ by a homeomorphism isotopic to the identity.
\end{prop}

\begin{proof}
A nontrivial finite order homeomorphism of $X$ cannot be homotopic to the identity. Indeed, by Theorem~\ref{thm:CircleAction} if $f$ is homotopic to the identity and of order $n$, it has a lift $\tilde{f}$ such that $\hat{f}$ fixes $S^1$ pointwise. The map $\hat{f}^n$ is a deck transformation, hence, M\"{o}bius, fixing $S^1$ pointwise. Thus, $\hat{f}^n$ is the identity on $D^2$.  A finite-order homeomorphism $\hat{f}$ of $D^2$ fixing $S^1$ is the identity by a theorem of \Kerekjarto\ (see \cite[Cor.~4.3]{Kolev06}).
It follows that the natural map $G \to \text{Mod}^{\pm}(X)$ is injective.
Therefore, the function $\phi\colon G \to G'$ assigning to each member of $G$ the unique isometry of $X$ in its homotopy class is a group isomorphism (see \cite[\S 12.1]{FarbMargalit_Primer}).


Let $\tilde{G}$ and $\tilde{G}'$ be the groups of homeomorphisms of $\Hyp^2$ given by lifting elements of $G$ and $G'$.  
There is a natural isomorphism $\tilde{\phi}\colon\tilde{G} \to \tilde{G}'$ that equals the identity on their common finite index subgroup $H=\pi_1(X)$.
The topological actions of $\tilde{G}$ and $\tilde{G}'$ on $\Hyp^2$ extend to actions on the disc $D^2 = \Hyp^2 \cup S^1$ that agree on $S^1$ by Theorem~\ref{thm:CircleAction}. Since $\tilde{G}$ contains the convergence group $H$ with finite index, $\tilde{G}$ is also a convergence group.
It follows from Martin--Tukia \cite[Cor.~4.5]{MartinTukia88} that there is a homeomorphism $\hat{f}$ of $D^2$ conjugating $\tilde{G}$ to $\tilde{G}'$ such that $\hat{f}$ restricts to the identity on $S^1$.
In particular, $\hat{f}$ is $H$--equivariant, so it induces a homeomorphism $f$ of $X$ conjugating $G$ to $G'$.
By Theorem~\ref{thm:CircleAction}, the map $f$ is isotopic to the identity.
\end{proof}

The next theorem roughly states that if $G$ is virtually a geometrically finite Kleinian group and $G$ acts on $S^2$, then $G$ is itself Kleinian.
For emphasis, we remind the reader of our convention that a Kleinian group is a discrete subgroup of $\Isom(\Hyp^3)$, whose elements do not need to preserve orientation.  Similarly, a conformal map is not required to preserve orientation.

\begin{thm}[Virtually Kleinian]
\label{thm:virtuallyKleinian}
Suppose $G$ acts faithfully as a convergence group on $S^2$, and $G$ has a finite index subgroup $H$ whose action on $S^2$ is topologically conjugate to a geometrically finite Kleinian action on $\widehat{\C}$.
Then the action of $G$ on $S^2$ is itself topologically conjugate to a geometrically finite Kleinian action on $\widehat{\C}$.
\end{thm}

Note: This conclusion follows from Mostow--Prasad Rigidity in the case that the limit set equals $S^2$, but requires more work in the general case.

\begin{proof}
Any convergence group on $S^2$ topologically conjugate to $G$ or $H$ must again be geometrically finite (see Section~\ref{sec:Convergence}).
Replacing $H$ with a finite index subgroup, we may assume that $H$ is torsion free, orientation preserving, and normal in $G$.
If $H \to \Gamma_H$ is any geometrically finite Kleinian representation topologically conjugate to the action of $H$ on $S^2$, then the conformal boundary $\Omega(\Gamma_H)/\Gamma_H$ is a finite type Riemann surface, \emph{i.e.}, it may be obtained from a compact Riemann surface by removing a finite set of points.

We produce a new representation $\Gamma'_H$ of $H$ compatible with the action of $G$ as follows. The given map $S^2 \to \widehat{\C}$ inducing $H \to \Gamma_H$ topologically conjugates the action of $G$ on $S^2$ to an action by homeomorphisms on $\widehat{\C}$. The induced action of $G/H$ on the Teichm\"{u}ller space of the surface $\Omega(\Gamma_H)/\Gamma_H$ has a fixed point by Nielsen realization \cite{Kerckhoff_NielsenRealization}, which determines a Kleinian representation $H\to \Gamma'_H$ topologically conjugate to $\Gamma_H$ (see Marden \cite[Thm.~3.1]{Marden07_Deformations}). 
The homeomorphism $S^2 \to \widehat{\C}$ inducing $H \to \Gamma'_H$ gives rise to an action $\rho\colon G \to \Homeo(\widehat{\C})$ that extends the Kleinian action $\Gamma'_H$ and that descends to an action of $G/H$ by homeomorphisms on $\Omega(\Gamma'_H)/\Gamma'_H$ such that each homeomorphism is homotopic to a conformal map.

By Proposition~\ref{prop:Isotopic}, there exists a homeomorphism $f$ of $\Omega(\Gamma'_H)/\Gamma'_H$ isotopic to the identity that topologically conjugates the action of $G/H$ to an action by conformal maps. The map $f$ lifts to a $\Gamma'_H$--equivariant homeomorphism of $\Omega(\Gamma'_H)$, which extends to a homeomorphism $\hat{f}$ of $\widehat{\C}$ that equals the identity on $\Lambda(\Gamma'_H)$ by Theorem~\ref{thm:MardenIsomorphism}(\ref{item:MardenIdentity}).
The map $\hat{f}$ topologically conjugates $\rho$ to an action $\rho'\colon G \to \Homeo(\widehat{\C})$ that extends the Kleinian action $\Gamma'_H$ and whose restriction to $\Omega(\Gamma'_H)$ is conformal.
By Theorem~\ref{thm:MardenIsomorphism}(\ref{item:MardenConformal}), the map $\rho'$ must be conformal on $\widehat{\C}$; in other words, $\rho'$ is Kleinian.
\end{proof}

\begin{thm}[Pinching curves]
\label{thm:Ohshika}
Let $\rho\colon G \to \Gamma \le \Isom(\Hyp^3)$ be a geometrically finite representation.
Let $\{\alpha\}$ be a finite collection of disjoint, nonparallel embedded $1$--orbifolds in the conformal boundary $\Omega(\Gamma)/\Gamma$ represented by a family $P_\ell$ of maximal loxodromic subgroups of\/ $\Gamma$.
In the conformal boundary, let $\{N_\alpha\}$ be pairwise disjoint regular neighborhoods of the $\alpha$'s.

Let $S = \widehat{\C}/\!\!\sim$ be the quotient formed by collapsing to a point each elevation of each $N_\alpha$.
Then the induced action of $G$ on $S$ is topologically conjugate to a geometrically finite Kleinian representation $\rho'\colon G \to \Gamma'$ whose parabolic elements are those of\/ $\Gamma$ along with the conjugacy classes of members of $P_\ell$.
In particular, the Kleinian group $\Gamma'$ is covered by $\Gamma$.
\end{thm}

The case where $G$ is orientation preserving and torsion free is due to Ohshika \cite{Ohshika98}.  

\begin{proof}
The quotient space $S$ is homeomorphic to $S^2$ by Moore's decomposition theorem (Theorem~\ref{thm:Moore}).
Let $H \le G$ be an orientation-preserving torsion-free subgroup of finite index.
By Ohshika \cite{Ohshika98}, the action of $H$ on $S$ is topologically conjugate to a geometrically finite Kleinian representation as in the theorem.  The result now follows from Theorem~\ref{thm:virtuallyKleinian}.
\end{proof}

Notice that the covering $\widehat{\C} \rightarrow \widehat{\C}$ corresponding to the isomorphism $\Gamma\to\Gamma'$ is equivariantly homeomorphic to the covering $S^2_\infty \to S^2_\emptyset$ given in Theorem~\ref{thm:CoveringConvergence}.  
Thus, we record: 

\begin{prop}
\label{prop:unpinchpinch} 
Let $(G, \mathcal{Q}\sqcup \mathcal{P}_h)$ be a geometrically finite convergence group on $S^2$ such that each member of $\mathcal{P}_h$ is two-ended.  
If the unpinched action on $S^2$ given by Theorem~\ref{thm:CoveringConvergence} is topologically conjugate to a geometrically finite Kleinian action $\Gamma$ on $\widehat{\C}$ then the pinched Kleinian action $\Gamma'$ given by Theorem~\ref{thm:Ohshika} is topologically conjugate to the original action of $G$ in $S^2$.
\end{prop}

\begin{proof}
If we unpinch, as in Theorem~\ref{thm:unpinch}, apply a topological conjugacy, and then pinch, as in Theorem~\ref{thm:Ohshika}, the result is a topological conjugacy.  Note that if the topological conjugacy of unpinched structures results in a Kleinian group, the final pinched structure is also Kleinian. Thus, the original group action is topologically conjugate to a Kleinian group. 
\end{proof}

A group is \emph{virtually compact special} if it has a finite index subgroup that is the fundamental group of a compact special cube complex.
We will use the following result, which is a minor variation of a theorem due to Wise.

\begin{thm}[\cite{Wise_QCHierarchy}, Thm.~17.14]
\label{thm:VCS}
Let $H$ act properly and isometrically on $\Hyp^3$ with noncompact quotient. \textup{(}We do not assume that the action is faithful.\textup{)}
Then $H$ is virtually compact special.
\end{thm}

The theorem also holds for groups acting cocompactly by \cite{KahnMarkovic12,BergeronWise12,Wise_QCHierarchy,Agol_VirtualHaken}, but that stronger result is not needed in this paper.

\begin{proof}
By a theorem of \Haissinsky--Lecuire, if $H$ is a finitely generated group that maps onto a Kleinian group with a finite kernel $N$, then $H$ and $H/N$ are commensurable \cite[Thm.~1.3]{HaissinskyLecuire_3manifold}.
In the present setting, the given group $H$ acts on $\Hyp^3$ with a finite kernel $N$. Therefore, $H$ is virtually the fundamental group of a compact atoroidal $3$--manifold with nonempty boundary.
It follows from \cite[Thm.~17.14]{Wise_QCHierarchy} that $H$ is virtually compact special.
\end{proof}

The Virtually Compact Special Theorem will be used together with the following result of Huang--Wise.

\begin{thm}[\cite{HuangWise_Stature}, Thm.~1.9]
\label{thm:HW_Stature}
Let $G$ be hyperbolic relative to finitely generated virtually abelian subgroups. Suppose $G$ splits as a finite graph of groups whose edge groups are relatively quasiconvex and whose vertex groups are virtually compact special. Then each relatively quasiconvex subgroup of a vertex group is separable in $G$.
In particular, $G$ is residually finite.
\end{thm} 

Elementary subgroups of a relatively hyperbolic group are relatively quasiconvex, so this theorem applies to splittings with elementary edge groups.

\section{Combination Theorems}
\label{sec:Combination}

In this section, we discuss three types of combinations of convergence groups. 
If $G$ acts as a geometrically finite convergence group on $S^2$ and $G$ admits a splitting over elementary subgroups with the action of each vertex group topologically conjugate to a Kleinian action, then under certain conditions we conclude that the action of $G$ of $S^2$ is topologically conjugate to a  Kleinian group. We do this for several different types of amalgamations. 

In Subsection~\ref{subsection:finite}, we show how to amalgamate convergence groups acting on $S^2$ over finite subgroups, following \cite{MartinSkora89}. In Subsection~\ref{subsection:parabolic}, we study combinations of convergence groups along peripheral groups---see 
Proposition~\ref{prop:PeripheralSplitting}.  Finally, in Subsection~\ref{subsection:loxodromics}, we study combinations of convergence groups along loxodromic subgroups---see 
Proposition~\ref{prop:LoxodromicDecomposition}.  This last part relies on Thurston's Hyperbolization Theorem.

These results---perhaps of interest in their own right---are used in Section~\ref{sec:ElementaryHierarchy} to ``go up the hierarchy'', a key ingredient in the proof of Theorem~\ref{thm:main}. 

\subsection{Combinations over finite groups} \label{subsection:finite}

We study groups with virtualy abelian peripheral subgroups and disconnected boundary such that the stabilizer of each component of the boundary is Kleinian.

\begin{prop}[\emph{cf.}\ \cite{MartinSkora89}, Thm.~4.3]
\label{prop:FiniteDecomposition}
Suppose a group $G$ acts as a faithful geometrically finite convergence group on $S^2$ with each parabolic subgroup virtually abelian.
If the stabilizer of each nontrivial component of $M=\Lambda G$ is topologically conjugate to a geometrically finite Kleinian group, then $G$ itself is topologically conjugate to a geometrically finite Kleinian group.
\end{prop}

A special case of this result is proved in \cite{MartinSkora89}, when each component of $\Lambda G$ is point or a circle.
Our proof follows essentially the same reasoning, but we provide the details since they will be used several times below.

\begin{proof}
By Theorem~\ref{thm:virtuallyKleinian}, it suffices to prove the theorem in the special case that $G$ is orientation preserving, which we now assume.
A theorem of Osin \cite{Osin06} states that a relatively hyperbolic group is finitely presented whenever its peripheral subgroups are finitely presented. Therefore, $G$ is finitely presented, and an accessible group by \cite{Dunwoody_Accessibility}.

The Decomposition Theorem of Martin--Skora \cite[Thm.~3.2]{MartinSkora89} states the following.  Suppose $G \le \Homeo(S^2)$ is an orientation preserving convergence group such that $G$ is accessible.  Then there exists a $G$--invariant family $\mathcal{C}$ of disjoint simple closed curves in $\Omega(G)$ that lie in finitely many $G$--orbits such that for each nontrivial component $U$ of $S^2 \setminus \bigcup \mathcal{C}$ the limit set of $\Stab(U)$ is connected. Such a family of curves in the ordinary set is necessarily a null family by the convergence property of $G$.

To prove the proposition, we induct on the number of $G$--orbits of curves in $\mathcal{C}$.  In the base case, $\mathcal{C}$ is empty and $M$ is connected, so there is nothing to prove.
Now suppose $\mathcal{C}$ is nonempty, and choose any loop $\beta \in \mathcal{C}$.  
Consider the graph $T$ with one vertex for each nontrivial component of $S^2 \setminus \bigcup_{g\in G} g(\beta)$, and an edge between adjacent components. 
This graph is a tree: it has no circuit by the planarity of $\mathcal{C}$, and it is connected since any two points in $\Omega(G)$ are separated by finitely many curves.   
As $T$ has only one orbit of edges, it determines a splitting of $G$ as an amalgam or HNN extension over a finite cyclic group $\Z_n$ for some $n$. 

We first consider an amalgam $G=G_1 *_{\Z_n} G_2$.  Let $F_1$ and $F_2$ be the components of $S^2 \setminus \bigcup_{g\in G} g(\beta)$ adjacent to $\beta$, so that $F_i$ is stabilized by $G_i$ for $i=1,2$ and, thus, $\Lambda(G_i) \subset F_i$.
The vertex stabilizers in $T$ have lower complexity, so they satisfy the proposition by induction.
Thus, there are homeomorphisms $\eta_i \colon S^2 \to \widehat{\C}$ topologically conjugating each $G_i < \Homeo(S^2)$ to a geometrically finite Kleinian group $\Gamma_i < \PSL(2,\C)$.
Since $\Lambda(G_i) \subset F_i$, the Jordan curve $\eta_i(\beta)$ bounds a disc in $\widehat{\C}$ that lies in $\Omega(\Gamma_i)$.
Since all closed discs in a connected surface are ambiently isotopic, there exists a $\Gamma_i$--equivariant homeomorphism of $\Omega(\Gamma_i)$ taking $\eta_i(\beta)$ to a round circle.  Any $\Gamma_i$--equivariant homeomorphism  of $\Omega(\Gamma_i)$ continuously extends to a $\Gamma_i$--equivariant homeomorphism of $\widehat{\C}$ by Theorem~\ref{thm:MardenIsomorphism}.
In other words, we may choose the homeomorphisms $\eta_i$ to be orientation preserving maps such that $\eta_1 \big| \beta = \eta_2 \big| \beta$ and such that the common image $\eta_1(\beta) = \eta_2(\beta)$ is a round circle of $\widehat{\C}$.

The Kleinian groups $\Gamma_1$, $\Gamma_2$, and $J=\Gamma_1 \cap \Gamma_2 = \Z$ together with the circle $\eta_i(\beta)$ satisfy the hypothesis of Maskit's first combination theorem \cite{Maskit88}.
Therefore, the group $\Gamma = \langle \Gamma_1,\Gamma_2 \rangle$ is a geometrically finite Kleinian group and the natural map $\phi\colon G = G_1 *_{\Z} G_2 \to \Gamma$ is an isomorphism.

To see that $G$ and $\Gamma$ are topologically conjugate, we recall that the family of circles $G(\beta)$ is null.
The sphere $S^2$ is a union of translates of the closures of $F_1$ and $F_2$ together with a $0$--dimensional set $L'$ consisting of points that are each obtained as the nested intersection of a shrinking sequence of discs bounded by $G$--translates of $\beta$.  The set $L'$ naturally corresponds to the set of ends of $T$. The limit set $\Lambda(G)$ is equal to the union of limit sets of the vertex stabilizers of $T$ together with the set $L'$.

The combination theorem gives a similar decomposition of $\widehat{\C}$.
The family of circles $\Gamma(\eta_i \beta)$ is null (see \cite[Thm.~VII.C.2(v)]{Maskit88}).  Furthermore, $\widehat{\C}$ is a union of translates of the closures of $\eta_1(F_1)$ and $\eta_2(F_2)$ together with a $0$--dimensional set $K'$ of points that are each obtained as the nested intersection of a shrinking sequence of discs bounded by $\Gamma$--translates of $\eta_i(\beta)$ (see \cite[Thm.~VII.C.2(vi)]{Maskit88}). Thus, $K'$ also corresponds to the ends of $T$.

Recall that if $A$ is a dense subspace of $X$ and $Y$ is a complete metric space, any continuous map $f\colon A\to Y$ uniquely extends to a continuous map $X\to Y$ iff the oscillation of $f$ vanishes at each point of $X\setminus A$ (see \cite[Lem.~4.3.16]{Engelking_GT}).
We define $\eta\colon S^2 \setminus L'\to \widehat{\C} \setminus K'$  as follows.
If $x \in S^2 \setminus L'$, then $x \in g(\bar{F}_i)$ for some $g \in G$ and $i=1$ or $2$; in which case, we define $\eta(x) = \phi(g) \bigl( \eta_i (g^{-1}x) \bigr)$.
The oscillation of $\eta$ vanishes at each point of $L'$ since $\Gamma(\eta_i\beta)$ is null, and the oscillation of $\eta^{-1}$ vanishes at each point of $K'$ since $G(\beta)$ is null.
By completeness, $\eta$ and $\eta^{-1}$ extend continuously to maps $S^2 \to \widehat{\C}$ and $\widehat{\C} \to S^2$.
These extensions respect the natural equivariant identifications of $L'$ and $K'$ with the ends of the tree $T$.  Thus, the extensions are mutually inverse equivariant homeomorphisms.

The proof in the HNN extension case follows by similar reasoning using Maskit's second combination theorem \cite{Maskit88}.
\end{proof}

\subsection{Combinations over parabolic subgroups}
\label{subsection:parabolic}

In this subsection, we study a group $G$ with connected boundary and a maximal peripheral splitting of it corresponding to its cut point tree.  We assume (as given) that each vertex group of this splitting (thought of as a group of homeomorphisms of $S^2$) is topologically conjugate to a Kleinian group on $\widehat{\C}$.  We use several results from planar topology that are discussed in Section~\ref{sec:Planar}.

\begin{defn}[Cut point tree]
\label{defn:CutPointTree}
Suppose $(G,\mathcal{P})$ is relatively hyperbolic with connected Bowditch boundary $M$.
A subset $C$ of $M$ is a \emph{cyclic element} if $C$ consists of a single cut point or contains a non-cutpoint $p$ and all points $q$ that are not separated from $p$ by any cut point of $M$.
(See \cite[\S 52]{Kuratowski_VolII} for the classical structure theory of cyclic elements of a Peano continuum.)
Consider the bipartite graph $T$ with vertex set $\mathcal{V}_0\sqcup \mathcal{V}_1$, where $\mathcal{V}_0$ is the family of nontrivial cyclic elements of $M$ and $\mathcal{V}_1$ is the set of all cut points of $M$. Vertices $B \in \mathcal{V}_0$ and $v\in \mathcal{V}_1$ are joined by an edge if $v \in B$.

The boundary $M = \boundary(G,\mathcal{P})$ is a Peano continuum, and the graph $T$ is a simplicial tree, known as the \emph{cut point tree} of $M$.  The stabilizer of each $\mathcal{V}_1$--vertex is a maximal parabolic subgroup and the stabilizer of each $\mathcal{V}_0$--vertex is a nonelementary relatively quasiconvex subgroup of $(G,\mathcal{P})$.
The claims above are due to Bowditch in the case that each peripheral subgroup is virtually abelian \cite{Bowditch_Peripheral} (see \cite{DasguptaHruska_LC} for the general case).
\end{defn}

\begin{lem}[One direction lemma] 
\label{lem:OneDirection}
Suppose $(G,\mathcal{P})$ is relatively hyperbolic with connected boundary $M=\boundary(G,\mathcal{P})$ such that each $P \in \mathcal{P}$ is virtually abelian.
Suppose further that the boundary $M$ is planar and the action of $G$ on $M$ extends to a convergence group action on $S^2$ with limit set $M$.

Let $T$ be the cut point tree of $M$.
Fix a $\mathcal{V}_1$--vertex $v_P$ of $T$ with stabilizer $J$.
For each edge incident to $v_P$, the edge stabilizer is virtually infinite cyclic and all translations in edge groups incident to $v_P$ are parallel in the Euclidean plane.
Furthermore, 
let $Z_P$ be the subgroup of $P$ containing all translations in this common direction.  Then $Z_P$ is contained in every edge group, and the intersection of the edge groups---if there is more than one---is $Z_P$. 
\end{lem}

\begin{proof}
Suppose $P$ is a $\mathcal{V}_1$--vertex group with unique fixed point $p \in S^2$, which is a cut point of the boundary $M$. Each incident edge group $Q_C$ is the subgroup $P$ stabilizing a component $C$ of $M\setminus\{p\}$.
We need to show that each such stabilizer is virtually infinite cyclic.
The plane $S^2 \setminus \{p\}$ has either a Euclidean or hyperbolic metric invariant under $P$ by Theorem~\ref{thm:OrbifoldGeometrization}.  If this metric is hyperbolic, then $P$ must be virtually $\Z$.  However, any $2$--ended group of isometries of $\Hyp^2$ is topologically conjugate to a group of isometries of $\E^2$, so we may assume that the invariant metric is Euclidean.

The space $M-\{p\}$ is disconnected, and each component $C$ of $M\setminus \{p\}$ is unbounded in the plane since its closure in $M$ contains $p$. The stabilizer of $C$ is a subgroup of $P$ acting cocompactly on $C$.  The subgroup $Q_C$ leaving $C$ invariant must be infinite, but it cannot be virtually $\Z^2$, because then a finite tubular neighborhood of $C$ would fill the entire Euclidean plane and would contain a $Q$--equivariant copy of the Cayley graph of $\Z^2$.  
A connected $\Z^2$--invariant subset of the plane cannot contain any unbounded complementary components, but the other components of $M\setminus \{p\}$ must also be unbounded in the plane.
The only possibility is that $Q_C$ is virtually $\Z$.

Note that the edges of $T$ incident to $v_P$ involve translations in parallel directions, since otherwise the corresponding components $C$ would intersect.

To see that every element of $Z_P$ is contained in each edge group $Q_C$, let $Z_C$ be the intersection $Z_P \cap Q_C$.
Choose a generator $\gamma$ for $Z_P$.  Suppose $Z_C$ is generated by $\gamma^k$ for some integer $k\ge 2$.
Then $\gamma$ maps $C$ to a disjoint component $C'$.
Choose a basepoint $x \in C$ and a path $\alpha$ in $C$ from $x$ to $\gamma^k(x)$.  The $Z_C$--translates of this path will form a connected set that separates the plane.  The vertical translation of this curve by $\gamma$ will intersect the original curve, contradicting the fact that $C$ and $C'$ are disjoint.
\end{proof} 

The proof of Proposition~\ref{prop:FiniteDecomposition} uses the Decomposition Theorem of Martin--Skora \cite{MartinSkora89}.
The following result is an analogue in the parabolic case. 

\begin{prop}[Parabolic decomposition]
\label{prop:ParabolicDecomposition}
Let $G < \Homeo(S^2)$ be a torsion-free, orientation preserving, geometrically finite convergence group. Let $\mathcal{P}$ be the family of all maximal parabolic subgroups of $G$, and suppose each $P \in \mathcal{P}$ is virtually abelian.
Suppose the limit set $M=\Lambda G$ is connected.
Then there exists a $G$--invariant family $\mathcal{C}$ of simple closed curves in $S^2$ with the following properties.
   \begin{enumerate}
   \item
   \label{item:FiniteOrbits}
   The curves of $\mathcal{C}$ lie in finitely many $G$--orbits.
   \item
   \label{item:CurvesAreNull}
   The family $\mathcal{C}$ is a null family.
   \item
   \label{item:Horocyles}
   Each curve in $\mathcal{C}$ is contained in $\Omega(G)$ except at one point, where it passes through a cut point of $M$.
   \item
   \label{item:PairwiseDisjoint}
   The lines $\bigset{\beta \cap \Omega(G)}{\beta \in \mathcal{C}}$ are pairwise disjoint.
   \item
   \label{item:SeparatedByCutPoint}
   Two points of $M$ are separated by a cut point of $M$ if and only if they are separated by a curve of $\mathcal{C}$.
   \item
   \label{item:CyclicElements}
   For each nontrivial component $U$ of $S^2 \setminus \bigcup \mathcal{C}$, the closure $\bar{U}$ contains a unique nontrivial cyclic element of $M$.
   \item
   \label{item:DualTree}
   The graph $T'$ dual to $\mathcal{C}$ is a tree.
   \end{enumerate}
\end{prop}

\begin{proof}
Let $\Delta$ be any component of $\Omega(G) = S^2 \setminus M$.
Then $\boundary\Delta$ is a Peano continuum by Theorem~\ref{thm:PlanarPeano}(\ref{item:PeanoTorhorst}).
If $\Stab_G(\Delta) \ne \Stab_G(\boundary\Delta)$ then $\boundary\Delta$ is the common frontier of distinct regions $\Delta$ and $g(\Delta)$ for some $g\in G$, which can only happen if $\boundary\Delta=S^1$ (\cite[Thm.~61.II.12]{Kuratowski_VolII}).
In this case, $\bar\Delta$ is a Jordan region and $\bar\Delta \cup g\bar\Delta = S^2$. It follows that $M=\boundary\Delta$ and $M$ has no cut points.

Without loss of generality, we will assume that $M$ has at least one cut point, so that $H = \Stab_G(\Delta) = \Stab_G(\boundary\Delta)$.
According to Theorem~\ref{thm:PlanarPeano}(\ref{item:PeanoNull}), the family of translates $G(\boundary\Delta)$ is null.  It follows from \cite[Prop.~2.3]{HPWpreprint} that $H$ is a relatively quasiconvex subgroup of $G$ with limit set $\Lambda H = \boundary\Delta$.
In particular, $H$ is finitely generated, so that the quotient $\Delta/H$ is a surface of finite topological type, \emph{i.e.}, a closed surface with finitely many points removed.
Each cut point $p \in \boundary\Delta$ is also a cut point of $M$ by Theorem~\ref{thm:PlanarPeano}(\ref{item:PeanoCutPoint}), so it is stabilized by a maximal parabolic subgroup of $H$. These cut points lie in finitely many $H$--orbits by relative quasiconvexity.

For each cut point $p$ of $\boundary \Delta$, let $P = \Stab_G(p)$ be the maximal parabolic subgroup of $G$ fixing $p$. 
A key point is that, as $P$ is torsion free and orientation preserving, the subgroup $P\cap H = \Stab_H(p)$ is equal to the infinite cyclic translation group $Z_P$ described in Lemma~\ref{lem:OneDirection}.  In particular, $p$ cuts $\boundary\Delta$ into exactly two components, and $p$ is a local cut point of $\bar\Delta$ of valence two since $\bar\Delta \setminus \{p\}$ is two-ended.  
Choose an embedded $Z_P$--invariant line $\ell_p$ in the two-ended space $\bar\Delta \setminus \{p\}$ so that the family $\mathcal{L}(\Delta)$ of all such lines $\{\ell_p\}$, one for each cut point $p$, is $H$--invariant.

Let $q\colon D^2 \to \bar\Delta$ be the \Caratheodory--Torhorst quotient map given by Theorem~\ref{thm:Torhorst}.
Then $q$ identifies points of $S^1$ in pairs such that no two pairs are linked in $S^1$.
By Proposition~\ref{prop:LiftingConvergence}, the convergence group action of $H$ on $\bar{\Delta}$ lifts to a convergence group action of $H$ on $D^2$.  
The action on $D^2$ is topologically conjugate to a Fuchsian action by Theorem~\ref{thm:MartinTukia}, so we may equivariantly identify $D^2$ with $\Hyp^2 \cup S^1$, and we may similarly identify $\Delta$ with the hyperbolic plane.
Choosing the lines of $\mathcal{L}(\Delta)$ to be hyperbolic geodesics, the nonlinking property of their endpoints in $S^1$ implies that the lines $\ell_p$ have pairwise distinct closures in $\bar\Delta$.

For each $\ell_p \in \mathcal{L}(\Delta)$, the closure $\bar\ell_p \subset \bar\Delta$ is an embedded circle in $S^2$ that lies in $\Delta$ except at one point where it passes through the corresponding cut point $p$.
Let $\mathcal{C}$ be the corresponding family of all such circles in $S^2$ associated to all components of $\Omega(G)$.
By construction, $\mathcal{C}$ satisfies (\ref{item:Horocyles}), (\ref{item:PairwiseDisjoint}), (\ref{item:SeparatedByCutPoint}), and (\ref{item:CyclicElements}).

To establish (\ref{item:FiniteOrbits}), recall that each family of lines  $\mathcal{L}(\Delta)$ has finitely many $H$--orbits, where $H=\Stab_G(\Delta)$.
It suffices to show that the components $\Delta$ of $\Omega(G)$ lie in finitely many $G$--orbits.
By a theorem of Gerasimov \cite{Gerasimov09}, the geometrically finite convergence group $G$ acts cocompactly on the space of distinct pairs of points of $M$.  The space of distinct pairs has an exhaustion by compact sets $\bigset{(a,b)}{d(a,b)\ge \epsilon}$ for all $\epsilon>0$.
Therefore, there exists $\epsilon>0$ such that any subset of $M$ of cardinality at least two has a $G$--translate with diameter at least $\epsilon$.
In particular, every component of $\Omega(G)$ has a $G$--translate whose boundary has diameter at least $\epsilon$.  There are only finitely many such components by Theorem~\ref{thm:PlanarPeano}(\ref{item:PeanoNull}).

We now verify (\ref{item:CurvesAreNull}).
Since the set of components of $\Omega(G)$ is null, and for each component $\Delta$ the lines of $\mathcal{L}(\Delta)$ lie in finitely many orbits, it suffices to prove nullity for the $H$--orbit of any single line $\ell_p \in \mathcal{L}(\Delta)$, where $H=\Stab(\Delta)$.
It is well-known in hyperbolic geometry that the translates of any geodesic under a proper isometric action on $\Hyp^2$ are a null family in the closure $D^2 = \Hyp^2\cup S^1$.
Thus, the orbit $H(\ell_p)$ in $\bar\Delta$ is a null family as well.

We now show (\ref{item:DualTree}).  The dual graph $T'$ has one vertex for each nontrivial component of $S^2 \setminus \mathcal{C}$ and an edge joining two vertices when they share a boundary circle in $\mathcal{C}$. 
To see that $T'$ is a tree, we first show that it is connected. Let $T$ be the cut point tree of $M$, described in Definition~\ref{defn:CutPointTree}.
The vertices of $T'$ are naturally identified with the $\mathcal{V}_0$--vertices of $T$.  If two $\mathcal{V}_0$--vertices are adjacent to the same $\mathcal{V}_1$--vertex, the corresponding cyclic elements intersect in a cut point of $M$.
Such vertices are joined by a finite edge path in $T'$ since the components meeting a cut point are cyclically ordered, as described in Lemma~\ref{lem:OneDirection}.
Since $T$ is connected, $T'$ is also.
To see that $T'$ has no cycles, note that each edge of $T'$ corresponds to a Jordan curve $\beta \in \mathcal{C}$ and its adjacent vertices are on opposite sides of $\beta$.
\end{proof}

\begin{prop}
\label{prop:PeripheralSplitting}
Let $G < \Homeo(S^2)$ be a faithful geometrically finite convergence group with $M = \Lambda G$ connected. Let $\mathcal{P}$ be the family of maximal parabolic subgroups of $G$, and suppose each $P \in \mathcal{P}$ is virtually abelian.

Suppose the stabilizer of each nontrivial cyclic element of $M$ is topologically conjugate to a geometrically finite Kleinian group on $\widehat{\C}$.
Then $G$ itself is topologically conjugate to a geometrically finite Kleinian group.
\end{prop}

\begin{proof}
In order to apply Proposition~\ref{prop:ParabolicDecomposition}, we must first pass to a torsion-free subgroup of finite index. The stabilizer $G_v$ of each $\mathcal{V}_0$--vertex $v$ of the cut point tree $T$ is topologically conjugate to a geometrically finite Kleinian group, so each $G_v$ is virtually compact special by \cite{Wise_QCHierarchy} (see Theorem~\ref{thm:VCS} for a detailed statement).
In particular, $G$ is residually finite by a theorem of Huang--Wise (Theorem~\ref{thm:HW_Stature}), so $G$ is virtually torsion free.  
By Theorem~\ref{thm:virtuallyKleinian} it suffices to prove the proposition with the additional hypothesis that $G$ is torsion free and orientation preserving.  We now work under that hypothesis.
Without loss of generality, we also assume that $M$ has at least one cut point.

Let $\mathcal{C}$ be the $G$--invariant family of simple closed curves in $S^2$ given by Proposition~\ref{prop:ParabolicDecomposition}, and let $T'$ be its dual tree.
Recall that the stabilizer of any component is the stabilizer of a nontrivial cyclic element of $M$.
In particular, each vertex group of $T'$ is topologically conjugate to a Kleinian group.

Since $T'$ has finitely many orbits of edges, we complete the proof by induction on the number of edge orbits.  
If $p$ is any cut point of $M$ with rank one stabilizer, any curve $\beta \in \mathcal{C}$ passing through $p$ determines a splitting of $G$ as an amalgam or an HNN extension with loxodromic stable letter.
On the other hand, if $p$ is a cut point of $M$ with rank two stabilizer in $G$, any curve $\beta \in \mathcal{C}$ through $p$ determines an HNN extension whose stable letter acts parabolically with fixed point $p$.  This last case turns a rank one cusp into a rank two cusp.
In all cases, the vertex groups of this splitting have lower complexity than $G$, so by induction they are already known to be Kleinian.

We first consider the case that $p$ has rank one stabilizer and $\beta$ corresponds to a splitting $G=G_1 *_{\Z} G_2$.
Let $G(\beta)$ be the union of all $G$--translates of the loop $\beta$.  Let $F_1$ and $F_2$ be the components of $S^2 \setminus G(\beta)$ adjacent to $\beta$, so that $F_i$ is stabilized by $G_i$ for $i=1,2$.
By the inductive hypothesis, there are homeomorphisms $\eta_i \colon S^2 \to \widehat{\C}$ topologically conjugating $G_i < \Homeo(S^2)$ to a geometrically finite Kleinian group $\Gamma_i \le \PSL(2,\C)$.
The rank-one parabolic fixed point $\eta_i(p)$ of $\Gamma_i$ is \emph{doubly cusped} in the sense that $\Omega(\Gamma_i)$ contains a pair of disjoint round open discs whose closures meet at $\eta_i(p)$ and whose $\Gamma_i$--translates are pairwise disjoint \cite[Thm.~VI.C.7]{Maskit88}.
By a theorem of Epstein, freely homotopic essential simple closed curves in the orientable surface $\Omega(\Gamma_i)/\Gamma_i$ are ambiently isotopic \cite[Thm.~2.1]{Epstein66}.
Therefore, there exists a $\Gamma_i$--equivariant homeomorphism of $\Omega(\Gamma_i)$ that takes $\eta_i(\beta)$ to a round circle in the doubly cusped region for $\eta_i(p)$.  Any $\Gamma_i$--equivariant homeomorphism  of $\Omega(\Gamma_i)$ continuously extends to a $\Gamma_i$--equivariant homeomorphism of $\widehat{\C}$ by Theorem~\ref{thm:MardenIsomorphism}.
In other words, we may choose the homeomorphisms $\eta_i$ to be orientation preserving maps so that $\eta_1 \big| \beta = \eta_2 \big| \beta$ and so that the common image $\eta_1(\beta) = \eta_2(\beta)$ is a round circle of $\widehat{\C}$.

The Kleinian groups $\Gamma_1$, $\Gamma_2$, and $J=\Gamma_1 \cap \Gamma_2 = \Z$ together with the circle $\eta_i(\beta)$ satisfy the hypothesis of Maskit's first combination theorem \cite{Maskit88}.
Therefore, the group $\Gamma = \langle \Gamma_1,\Gamma_2 \rangle$ is a geometrically finite Kleinian group and the natural map $\phi\colon G = G_1 *_{\Z} G_2 \to \Gamma$ is an isomorphism.

To see that $G$ and $\Gamma$ are topologically conjugate, we construct a homeomorphism using the tree structure. 
First recall that the family of circles $G(\beta)$ in $S^2$ is null by Proposition~\ref{prop:ParabolicDecomposition}(\ref{item:CurvesAreNull}).
The sphere $S^2$ is a union of translates of the closures of $F_1$ and $F_2$ and a $0$--dimensional set $L'$ consisting of nonparabolic points that are each obtained as the nested intersection of a sequence of discs whose boundary circles are $G$--translates of $\beta$.  The set $L'$ corresponds to the nonparabolic ends of $T'$, or equivalently the ends of the cut point tree $T$. The limit set $\Lambda(G)$ is equal to the union of limit sets of the vertex stabilizers of $T$ together with the set $L'$.
By the combination theorem, we have a similar decomposition of $\widehat{\C}$.
Using the same reasoning as in the proof of Proposition~\ref{prop:FiniteDecomposition}, it follows that $G$ and $\Gamma$ are topologically conjugate.

The proofs in the other two cases, involving HNN extensions, follow by similar reasoning using Maskit's second combination theorem \cite{Maskit88}.
\end{proof}

\subsection{Combinations over loxodromics}
\label{subsection:loxodromics}

In this subsection we study relatively hyperbolic group pairs $(G,\mathcal{P})$ whose boundary is connected with no cut points.  The JSJ decomposition over loxodromic subgroups has the following topological interpretation due to Haulmark--Hruska \cite{HaulmarkHruska_Canonical}. (The word hyperbolic case is due to Bowditch \cite{Bowditch98JSJ}.) See Definition~\ref{defn:JSJ} below for the notion of a JSJ decomposition over a family of subgroups.

\begin{defn}[Cut pair tree]
Let $M$ be a Peano continuum with no cut points.
A \emph{cut pair} is a pair of distinct points $\{a,b\}$ such that $M \setminus \{a,b\}$ is disconnected.
A cut pair $\{x,y\}$ is \emph{inseparable} if its points are not separated by any other cut pair. 
Let $x$ be a local cut point of $M$. The \emph{valence} of $x$ in $M$ is the number of ends of $M \setminus \{x\}$. Such a cut pair $\{x,y\}$ in $M$  is \emph{exact} if the number of components of $M \setminus \{x,y\}$ is equal to the valence of both $x$ and $y$. 
(Since $M$ is locally connected, the number of components of $M \setminus \{x,y\}$ is always finite.)
Let $Z$ denote the union of all inseparable exact cut pairs of $M$. Declare two points of $M\setminus Z$ to be equivalent if they are not separated by any inseparable exact cut pair. The closure of an equivalence class containing at least two points is a \emph{piece}.

Consider the bipartite graph $T$ with vertex set $\mathcal{V}_0 \sqcup \mathcal{V}_1$, where $\mathcal{V}_0$ is the set of pieces and $\mathcal{V}_1$ is the set of all inseparable exact cut pairs of $M$.  Vertices $B\in \mathcal{V}_0$ and $v \in \mathcal{V}_1$ are joined by an edge if $v \subset B$.
If the Peano continuum $M$ without cut points is the boundary of a relatively hyperbolic pair $(G,\mathcal{P})$, then the graph $T$ is a simplicial tree, known as the \emph{cut pair tree} of $M$. 
The stabilizer of each $\mathcal{V}_1$--vertex is a maximal two-ended nonparabolic subgroup, and the stabilizer of each $\mathcal{V}_0$--vertex is relatively quasiconvex.
Furthermore, the cut pair tree is equal to the canonical JSJ tree for splittings of $G$ over two-ended subgroups relative to $\mathcal{P}$.
These assertions are due to Bowditch in the hyperbolic case and Haulmark--Hruska in the general case \cite{Bowditch98JSJ,HaulmarkHruska_Canonical}.
\end{defn}

\begin{defn}[Regularity]
\label{defn:Regular}
Let $(G,\mathcal{P})$ be relatively hyperbolic.
An action of $G$ on a tree $T$ is \emph{regular} if each two-ended vertex group $G_v$ is infinite cyclic and stabilizes each edge $e$ adjacent to $v$.
\end{defn}

We use regularity as follows. Assume $M=\boundary(G,\mathcal{P})$ is connected with no cut points.
Let $T$ be the canonical JSJ tree for splittings over two-ended subgroups relative to $\mathcal{P}$.  Then the action is regular if each two-ended vertex group $L$ of $T$ is infinite cyclic and stabilizes each component of $M \setminus \Lambda L$.

\begin{prop}
\label{prop:LoxodromicDecomposition}
Let $G \le \Homeo(S^2)$ be a faithful geometrically finite convergence group whose limit set $M = \Lambda G$ is a Peano continuum with no cut points, and let $T$ be the cut pair tree of $M$.

For each $v \in \mathcal{V}_0(T)$, suppose $G_v$ is topologically conjugate to a geometrically finite Kleinian group on $\widehat{\C}$.
Then $G$ itself is topologically conjugate to a geometrically finite Kleinian group.
\end{prop}

\begin{proof}
As in the proof of Proposition~\ref{prop:PeripheralSplitting}, the group $G$ is residually finite.
Each nonelementary vertex group has a torsion-free, orientation preserving subgroup of finite index.
Furthermore, each two-ended vertex group $L$ has an infinite cyclic subgroup of finite index that stabilizes each of the finitely many components of $M \setminus \Lambda L$.
Thus, by residual finiteness, $G$ has a torsion-free, finite index subgroup whose action on $T$ is regular and whose action on $S^2$ is orientation preserving.
By Theorem~\ref{thm:virtuallyKleinian}, it suffices to prove the theorem for such a subgroup, so we will work under the hypothesis that $G$ itself has these properties.
Without loss of generality, we also assume $M$ is not homeomorphic to $S^1$, since the cut pair tree is trivial in that case.

Let $\Delta$ be any component of $\Omega(G) = S^2 \setminus M$, and let $H=\Stab_G(\Delta)$.  As explained in the proof of Proposition~\ref{prop:ParabolicDecomposition}, since $M\ne S^1$, we must have $H=\Stab_G(\boundary\Delta)$.
Since $M$ has no cut points, $\boundary\Delta$ is a simple closed curve in $S^2$ and $\bar\Delta$ is a Jordan region by Theorem~\ref{thm:PlanarPeano}(\ref{item:PeanoCutPoint}). 
According to Theorem~\ref{thm:MartinTukia}, the action of $H$ on $\Delta$ is topologically conjugate to a Fuchsian action, so we may equivariantly identify $\bar\Delta$ with $\Hyp^2 \cup S^1$.
By Theorem~\ref{thm:PlanarPeano}(\ref{item:PeanoNull}), the family of translates $G(\boundary\Delta)$ is null. So $H$ is relatively quasiconvex with limit set $\boundary\Delta$ by \cite[Prop.~2.3]{HPWpreprint}.
Thus, $H$ is finitely generated and $\Delta/H$ is a finite area hyperbolic surface.
As in the proof of Proposition~\ref{prop:ParabolicDecomposition}, the components $\Delta$ of $\Omega(G)$ lie in finitely many $G$--orbits.

Let $\{x,y\}$ be the inseparable exact cut pair of $M$ stabilized by a two-ended vertex group $L$ of $G$.
By exactness, $M \setminus \{x,y\}$ has a finite number $m=m(x,y)$ of components, each stabilized by $L$, and the pair $\{x,y\}$ lies in the boundary of precisely $m$ components of $\Omega(G)$, each stabilized by $L$.  Let $\Delta$ be any such component, and let $H=\Stab_G(\Delta)$. 
The loxodromic action of $L$ on $\Delta$ has a hyperbolic geodesic axis $\ell$ joining the pair of points $\{x,y\}$. Let $\mathcal{L}(\Delta)$ be the $H$--equivariant family of all such geodesics corresponding to inseparable exact cut pairs in $\boundary\Delta$.  By inseparability, these lines have pairwise distinct closures in $\bar\Delta$.
Since $\Delta/H$ is a finite area surface, $\mathcal{L}(\Delta)$ has only finitely many $H$--orbits of lines.
The closures in $\bar\Delta$ of these lines form a null family since the action on $\bar\Delta$ is Fuchsian.

Let $\mathcal{C}$ be the family of all such compact arcs in $S^2$ associated to all components $\Delta$, and all inseparable exact cut pairs.   Then $\mathcal{C}$ is also a null family.
The cut pair tree $T$ may be recovered from $\mathcal{C}$ as follows.
The $\mathcal{V}_0$ vertices are in one-to-one correspondence with nontrivial components $U$ of $S^2 \setminus \bigcup \mathcal{C}$, and the $\mathcal{V}_1$--vertices are the two-point sets $\{x,y\}$ of endpoints of the arcs of $\mathcal{C}$.
Suppose vertices $v=U$ and $w=\{x,y\}$ are adjacent in $T$. If $\{x,y\}$ cuts $M$ into $m=m(w)$ components, then $\mathcal{C}$ has exactly $m$ arcs with endpoints $\{x,y\}$.
Two of these arcs lie on the boundary of $U$. Their union $\beta(v,w)$ is a simple closed curve that lies in the ordinary set of $G$, except at two points where it passes through $x$ and $y$.

For each $v \in \mathcal{V}_0$, the limit set $\Lambda(G_v)$ is the intersection of $M$ with the closure of the associated component of 
$S^2 \setminus \bigcup \mathcal{C}$.  Choose a geometrically finite Kleinian group $\Gamma_v$ topologically conjugate to $G_v$. The $3$--manifold $N_v = \mathbb{H}^3 \cup \Omega(\Gamma_v) / \Gamma_v$ has a natural structure as an orientable pared $3$--manifold $(M_v,P_v)$.
For each vertex $w$ adjacent to $v$, the loop $\beta(v,w)$ projects to a pair of homotopic simple closed curves in $\boundary M_v \setminus P_v$ that bound an incompressible annulus $A=A(v,w)$.
Letting $w$ vary over all vertices adjacent to $v$ gives finitely many pairwise disjoint incompressible annuli.
If $Q_v$ denotes the union of these annuli, then $(M_v, P_v \cup Q_v)$ is a pared $3$--manifold.

For each $w \in \mathcal{V}_1$ corresponding to a cut pair $\{x,y\}$, the cyclic group $\Gamma_w$ is generated by an orientation preserving loxodromic map fixing $x$ and $y$. Such a map is topologically conjugate to the map $f(z)=2z$ on $\widehat{\C}$ by Theorem~\ref{thm:OrbifoldGeometrization}.
The associated $3$--manifold $N_w = \Hyp^3\cup (\C \setminus\{0\}) \big/ \Gamma_w$ for $\Gamma_w=\langle f \rangle$ is a solid torus.
The $m$ arcs of $\mathcal{C}$ that meet the limit set of $G_w$ project to $m$ parallel simple closed curves on the boundary of this solid torus, which cut the boundary into $m$ annuli $A(w,v)$, one for each vertex $v$ adjacent to $w$.

Form a graph of spaces from the manifolds $\{M_v\}$ and $\{N_w\}$, gluing along annuli in their boundaries so that if $v$ and $w$ are adjacent, $A(v,w)$ is glued to $A(w,v)$ by the natural orientation-reversing homeomorphism.
The result is a pared $3$--manifold $(M_G, P_G)$ by Proposition~\ref{prop:ParedCombination}.
By the Hyperbolization Theorem (Theorem~\ref{thm:Hyperbolization}), this pared $3$--manifold admits a complete geometrically finite hyperbolic structure. Let $\Gamma \le \PSL(2,\C)$ be the corresponding holonomy representation of $G$. This procedure gives representations of the vertex groups that agree on their common edge groups.

We now prove that $G$ and $\Gamma$ are topologically conjugate. Consider any $G$--invariant forest $F \subseteq T$, and let the tree $T'$ be a component of $F$.
Let $k$ denote the number of $G$--orbits of edges intersecting $T'$.
Since $T$ itself arises from the forest $F=T$, we will complete the proof by induction on $k$.
The inductive claim is that, for any such subtree $T'$, the group $G' = \Stab_G(T') \le \Homeo(S^2)$ is topologically conjugate to $\Gamma' = \Stab_\Gamma(T') \le \PSL(2,\C)$.

The base case, in which $k=0$ and $T'$ is a single vertex $v$, holds as a consequence of Marden's Isomorphism Theorem (see Corollary~\ref{cor:GeomFiniteManifolds}) applied to $\Gamma_v$ and the restriction of $\Gamma$ to $G_v$.  This step is necessary because the hyperbolic metric on $M_v$ provided by hyperbolization may not be the same as the one given by $\Gamma_v$ (although they are topologicially conjugate in $\widehat{\C}$). 

For the inductive step, choose a subtree $T'$ of $T$ as above. Fix a $\mathcal{V}_1$--vertex $w$ of $T'$ corresponding to a cut pair $\{x,y\}$.  If $w$ has valence one in $T'$, then $T'$ has a $G'$--invariant proper subtree with fewer orbits of edges than $T'$.  In this case, the conclusion follows immediately from the inductive hypothesis.

Now assume the valence of $w$ in $T'$ is at least two.  Then there exists a simple closed curve $\beta$ in $S^2$ consisting of the union of two arcs of $\mathcal{C}(T')$ with common endpoints $x$ and $y$.
By the regularity of the action on $T'$, the stabilizer $H$ of $\{x,y\}$ stabilizes each $\mathcal{V}_0$--vertex of $T'$ incident to the $\mathcal{V}_1$--vertex $\{x,y\}$.  Consequently, $H$ stabilizes each of the two discs in $S^2$ bounded by $\beta$.

In particular, $\beta$ determines a splitting of $G'$ as an amalgam or HNN extension with the single edge group $H$. The vertex groups of this splitting correspond to subtrees of $T'$ with fewer orbits of edges than $G'$.  The proof that $G'$ is topologically conjugate to $\Gamma'$ now follows from Maskit's combination theorem \cite{Maskit88} just as in the proofs of Propositions \ref{prop:FiniteDecomposition} and~\ref{prop:PeripheralSplitting} using the nullity of the family $G'(\beta)$.
\end{proof}

\section{The elementary hierarchy}
\label{sec:ElementaryHierarchy}

This section describes the elementary hierarchy of a relatively hyperbolic group pair, which breaks up the group into basic pieces. We describe this hierarchy in terms of its effect on the Bowditch boundary and show that it can be refined to a four-step procedure that is repeated indefinitely until each of the steps stabilizes.
Under the hypotheses of Theorem~\ref{thm:main}, we show in Proposition~\ref{prop:GoingDown} that the terminal pieces can be realized as Kleinian groups.

Then we use the combination results from Section~\ref{sec:Combination} to show that at each stage going up the hierarchy, we have a topological conjugacy to a Kleinian group.  In this section, we prove Theorems \ref{thm:generalcase} and~\ref{thm:puttogether} and Theorem~\ref{thm:main}. 

\begin{defn}[JSJ decompositions]
\label{defn:JSJ}
Consider a fixed group $G$ and the the family of all simplicial trees $T$ on which $G$ acts minimally without inversions.
A subgroup $H<G$ acts \emph{elliptically} on such a tree $T$ if $H$ fixes a vertex of $T$. Let $\mathbb{E}$ be any collection of subgroups of $G$ closed under conjugation and passing to subgroups, and let $\P$ be any family of subgroups of $G$. An \emph{$(\mathbb{E},\P)$--tree} is a tree $T$ such that each $P\in\P$ acts elliptically and each edge stabilizer is a member of $\mathbb{E}$.  An $(\mathbb{E},\P)$--tree is \emph{universally elliptic} if its edge stabilizers act elliptically on every $(\mathbb{E},\P)$--tree.
If $G$ acts on trees $T$ and $T'$, then $T$ \emph{dominates} $T'$ if there is a $G$--equivariant map $T\to T'$.

An $(\mathbb{E},\P)$--tree $T$ is a \emph{JSJ tree for splittings of $G$ over $\mathbb{E}$ relative to $\P$} if it satisfies the following universal properties:
\begin{enumerate}
    \item $T$ is universally elliptic among all $(\mathbb{E},\P)$--trees.
    \item $T$ dominates any other universally elliptic $(\mathbb{E},\P)$--tree.
\end{enumerate}
\end{defn}

We refer the reader to Guirardel--Levitt \cite{GuirardelLevitt_JSJ} for detailed background on JSJ decompositions of groups. As explained in \cite[Thm.~2.20]{GuirardelLevitt_JSJ}, if $(G,\mathcal{P})$ is relatively hyperbolic and $\mathbb{E}$ is the family of all elementary subgroups, a JSJ tree for $G$ over $\mathbb{E}$ relative to $\mathcal{P}$ always exists. 
If $G$ is one ended relative to $\mathcal{P}$, there is a special JSJ decomposition known as the tree of cylinders, which we call the \emph{canonical JSJ decomposition}, following Bowditch \cite{Bowditch98JSJ}.   

In this section, we consider relatively hyperbolic pairs $(G,\mathcal{P})$ with each member of $\mathcal{P}$ virtually abelian.  If any member of $\mathcal{P}$ is two ended, it can be removed from the peripheral structure via unpinching, as in Section~\ref{sec:unpinch}. Thus, we assume that all members of $\mathcal{P}$ are virtually abelian of rank at least two.

\begin{defn}[Hierarchies]
A \emph{hierarchy} for a group $H$ is a rooted tree of groups $\mathcal{H}$ with $H$ at the root such that the descendants of each group $L \in \mathcal{H}$ are the vertex groups of a graph of groups decomposition of $L$. A group $L \in \mathcal{H}$ is \emph{terminal} if $L$ has no descendants. A hierarchy is \emph{finite} if the underlying rooted tree is a finite tree.
\end{defn}

We recall the definition of an elementary hierarchy for a relatively hyperbolic group pair.  For simplicity, we focus on the case in which all peripheral subgroups are virtually abelian.

\begin{defn}
Let $(G,\mathcal{P})$ be relatively hyperbolic such that each member of $\mathcal{P}$ is virtually abelian.
An \emph{elementary JSJ hierarchy} for $(G,\mathcal{P})$ is a hierarchy $\mathcal{H}$ in which the given splitting of each $L \in \mathcal{H}$ is a JSJ decomposition of $L$ over virtually abelian subgroups relative to the family of maximal parabolic subgroups of $L$ that are not virtualy cyclic.
\end{defn}

Given a relatively hyperbolic $(G,\mathcal{P})$ with all peripheral subgroups virtually abelian, we obtain
an elementary hierarchy in four stages, discussed in detail below.  We give a quick summary here.  Each vertex group $L$ of the hierarchy is relatively quasiconvex in $(G,\mathcal{P})$, so it has an induced relatively hyperbolic structure (see, for instance, \cite[Prop.~3.4]{GuirardelLevitt15_AutRelHyp}).
First remove any two-ended groups from the induced peripheral structure of $L$, and then perform the following sequence of splittings of $L$ relative to this new peripheral structure.
Split using a JSJ decomposition over finite subgroups, then split using the canonical JSJ decomposition over parabolic subgroups, and finally split using the canonical JSJ decomposition over loxodromic subgroups.

In the case when there is no virtual two-torsion, this hierarchy is finite:

\begin{thm}[\cite{LouderTouikan17}, Cor.~2.7]
Let $(G,\mathcal{P})$ be relatively hyperbolic such that all members of $\mathcal{P}$ are virtually abelian. Assume that $G$ has a finite-index subgroup with no elements of order two.
Then any elementary JSJ hierarchy for $(G,\mathcal{P})$ is finite.
\end{thm}

\subsection{Topological interpretation of the elementary JSJ hierarchy} \label{sec:goingdowntop}

Assume that $(G, \mathcal{P})$ is hyperbolic relative to virtually abelian subgroups of rank at least two.
We now describe in detail how to produce an elementary hierarchy recursively. Each vertex group of this hierarchy is relatively hyperbolic and equipped with a preferred peripheral structure.

The purpose of this more detailed construction is to trace the effect of elementary splittings on the topology of the Bowditch boundary, and to better understand the convergence actions on $S^2$ of each vertex group.

Consider the following sequence of operations constructing a hierarchy by recursion on the number of levels.  Level $0$ of the hierarchy contains the group pair $(G,\mathcal{P})$.  If the first $k$ levels have been constructed previously,  we now describe how to produce the level $k+1$ descendants of each level $k$ vertex.

\begin{enumerate} 
\item[(0)] \textbf{(Terminal)}
Suppose $(H,\mathcal{P})$ is a vertex group at level $k$. If $(H,\mathcal{P})$ is either finite, parabolic, or Fuchsian, it is terminal in the hierarchy.
If $H$ has no splittings over elementary subgroups relative to $\mathcal{P}$ and the collection $\mathcal{P}$ does not contain a two-ended subgroup, we also consider $(H,\mathcal{P})$ to be terminal.
In all other cases, $(H,\mathcal{P})$ is non-terminal.

\item \label{unpinchZ} 
 \textbf{(Unpinch)}
Suppose $(H,\mathcal{P})$ is a non-terminal vertex group at level $k$ which has no splitting over elementary subgroups relative to $\mathcal{P}$ such that at least one member of $\mathcal{P}$ is a two-ended group.

We remove all $2$--ended groups from $\mathcal{P}$ to produce a new peripheral structure $(H,\mathcal{Q})$. We consider $(H,\mathcal{Q})$ to be the only descendant of $(H,\mathcal{P})$, so that the associated splitting of $(H,\mathcal{P})$ is trivial.

The effect of this removal on the Bowditch boundary is to unpinch each parabolic point that is a local cut point of valence two, as explained for general relatively hyperbolic groups in Theorem~\ref{thm:CollapsingBoundaries}.

When performing Step~\ref{unpinchZ}, the boundary $\boundary(H,\mathcal{P})$ may be disconnected, and there may be $2$--ended parabolic subgroups. After Step~\ref{unpinchZ}, the boundary $\boundary(H,\mathcal{Q})$ may be disconnected, but all parabolic subgroups in $\mathcal{Q}$ are virtually $\Z^n$ for $n\ge 2$. 

If $H$ acts as a geometrically finite convergence group on $S^2$, then we blow up the action on $S^2$ as described in 
Theorem~\ref{thm:CoveringConvergence}. When only $2$--ended subgroups are removed, this blowup may also be obtained using \cite[Thm.~4.2]{HPWpreprint}.

\item \label{splitfinite}
\textbf{(Split over finite subgroups)}
Suppose $(H,\mathcal{Q})$ is a non-terminal vertex group at level $k$ such that $\mathcal{Q}$ contains no two-ended groups and $H$ splits over a finite subgroup relative to $\mathcal{Q}$.
Then split $(H,\mathcal{Q})$ using a JSJ decomposition over finite subgroups relative to $\mathcal{Q}$.  We define the level $k+1$ descendants of $(H,\mathcal{Q})$ to be the vertex groups of this splitting.
If $L$ is any such descendant, then $L$ is considered to be hyperbolic relative to the family $\mathcal{Q}_L$ containing all infinite subgroups of the form $L \cap Q$ for $Q \in \mathcal{Q}$.
The infinite vertex groups of this splitting are precisely the family of stabilizers of nontrivial components of $\boundary(H,\mathcal{Q})$. There may also be finite vertex groups.

In Step~\ref{splitfinite}, all parabolic subgroups of $\mathcal{Q}$ are virtually $\Z^n$ for $n\ge 2$ and $\boundary(H, \mathcal{Q})$ may be disconnected. After Step~\ref{splitfinite} the boundary of each descendant vertex group $(L,\mathcal{Q}_L)$ is either empty or connected.  Still all maximal parabolic subgroups of $\mathcal{Q}_L$ are virtually $\Z^n$ for $n\ge 2$.

\item \label{splitperif}
\textbf{(Split over parabolic subgroups)}
Suppose that $(H,\mathcal{Q})$ is a non-terminal vertex group at level $k$ that does not split over finite subgroups relative to $\mathcal{Q}$ but that does split over some parabolic subgroup relative to $\mathcal{Q}$.  Suppose also that $\mathcal{Q}$ contains no two-ended groups.

Split $(H,\mathcal{Q})$ using the canonical JSJ decomposition over parabolic subgroups of $\mathcal{Q}$ relative to $\mathcal{Q}$. The tree corresponding to this splitting is equal to the cut point tree of the boundary, as described in Section~\ref{subsection:parabolic}.
If $L$ is a vertex group of this splitting, let $\mathcal{Q}_L$ be its induced peripheral structure, as described in the previous step. 

Before Step~\ref{splitperif}, the boundary of $(H, \mathcal{Q})$ is connected and could contain cut points, and all peripheral subgroups are virtually $\Z^n$ for $n\ge 2$. 
After Step~\ref{splitperif}, there are two types of vertex groups.  The most interesting case is a vertex group whose boundary $(L,\mathcal{Q}_L)$ is nontrivial and connected with no cut points (but it could contain cut pairs) and its peripheral subgroups may be virtually $\Z$ or may have higher rank.
The other case is a parabolic vertex group whose boundary is a single point.

\item \label{splitZ}
\textbf{(Split over loxodromic subgroups)}
Suppose that $(H,\mathcal{Q})$ is a vertex group at level $k$ that does not split over finite or parabolic subgroups relative to $\mathcal{Q}$ but that does admit a splitting over loxodromic subgroups relative to $\mathcal{Q}$.

Split $(H,\mathcal{Q})$ using the canonical JSJ decomposition over loxodromic subgroups of $\mathcal{Q}$ relative to $\mathcal{Q}$. The tree corresponding to this splitting is equal to the cut pair tree of the boundary, as described in Section~\ref{subsection:loxodromics}.
If $L$ is a vertex group of this splitting, let $\mathcal{Q}_L$ be its induced peripheral structure, as described above. 

Before Step~\ref{splitZ}, the boundary is connected with no cut points and its peripheral subgroups are either virtually $\Z$ or of higher rank.
After Step~\ref{splitZ}, each descendant vertex group is either $2$--ended, hanging Fuchsian, or rigid over $2$--ended groups. The rigid vertex groups might have disconnected boundary or cut points. Their peripheral subgroups may be either virtually $\Z$ or of higher rank.
\end{enumerate} 

The recursion above produces a hierarchy of elementary splittings that refines the Louder--Touikan elementary JSJ hierarchy. Under the assumption that the initial group $G$ is virtually without $2$--torsion, such a hierarchy terminates after a finite number of steps.

We now consider what happens when we apply the recursion above to a group that acts on $S^2$.  Suppose $(G,\mathcal{P})$ acts as a geometrically finite convergence group on $S^2$ whose maximal parabolic subgroups are the members of $\mathcal{P}$.
Suppose also that the members of $\mathcal{P}$ are each virtually abelian of rank two.
Each group pair in the hierarchy inherits a convergence group action on $S^2$ as described above---unpinching each time Step~\ref{unpinchZ} is applied.

\begin{prop}[Terminal groups are Kleinian] 
\label{prop:GoingDown}
Assume the hypotheses of Theorem~\ref{thm:generalcase}.  In particular, $G$ acts faithfully on $S^2$ and virtually has no $2$--torsion.
Assume either that $\partial(G, \mathcal{P})$ does not contain a \Sierpinski\ carpet or that the relative Cannon conjecture is true. 

Then any hierarchy for $(G,\mathcal{P})$ produced as above is finite.
The induced action of each terminal vertex group $G_v$ on $S^2$ is topologically conjugate to a geometrically finite Kleinian action of $G_v$ on $\widehat{\C}$.
More precisely, if the boundary contains no \Sierpinski\ carpet, then each terminal Kleinian group is either finite, rank-two parabolic, or Fuchsian.
\end{prop}

\begin{proof}
According to Louder--Touikan \cite{LouderTouikan17}, the elementary hierarchy is finite.  Each terminal vertex of the hierarchy carries a group pair $(H,\mathcal{Q})$ with a relatively hyperbolic structure with all peripheral subgroups virtually $\Z^2$.  The group pair $(H,\mathcal{Q})$ is either finite, parabolic, or nonelementary relatively hyperbolic with boundary that is connected with no cut points. 

If $\partial(H,\mathcal{Q})$ is a $2$--sphere, we are done by the relative Cannon conjecture, so we assume that $\Lambda H \ne S^2$. Thus, $\Lambda H$ has empty interior and, hence, has dimension at most one (\cite[Cor.~1.8.12]{Engelking_Dimensions}).
Each vertex group has at most one-dimensional connected boundary and admits no elementary splitting relative to peripheral subgroups.  By a theorem of Haulmark, its boundary is either the empty set, a single point, a circle, a \Sierpinski\ carpet, or a Menger curve \cite{HaulmarkRelHyp}.
But the Menger curve cannot arise, since it is not planar.

We claim that every terminal vertex group is topologically conjugate to a Kleinian group on $\widehat{\C}$. 
Indeed, if the boundary is empty the vertex group is finite, and the action on $S^2$ is conjugate to an orthogonal action by Theorem~\ref{thm:Kerekjarto}.
If the boundary is a single point $\xi$, the vertex group is virtually $\Z^2$ and the proper action on $S^2 \setminus \{\xi\}$ is conjugate to an isometric action on $\E^2$ by Theorem~\ref{thm:OrbifoldGeometrization}.
If the boundary is a circle, the convergence action on $S^2$ is conjugate to a Fuchsian action by Theorem~\ref{thm:MartinTukia}. 

Suppose for some terminal vertex, the boundary $\boundary(H,\mathcal{Q})$ is a carpet.
We first show that if $\boundary(G,\mathcal{P})$ does not contain a carpet, then $\boundary(H,\mathcal{Q})$ cannot be a carpet by induction on the height of the elementary hierarchy.
A connected component, a cut-point component, or a cut-pair component of a space with no carpets cannot contain a carpet.
Also, the boundary formed by unpinching a vertex group along $2$--ended subgroups cannot contain a \Sierpinski\ carpet. For if the unpinched boundary contains a carpet $\mathcal{S}$, then $\mathcal{S}$ contains a subcarpet $\mathcal{S}'$ not intersecting any peripheral circle of $\mathcal{S}$ (see Lemma~\ref{lem:Sierpinski} below), and $\mathcal{S}'$ embeds in the pinched boundary, a contradiction.

On the other hand, if the relative Cannon conjecture holds, we will show that the action of the carpet group $H$ on $S^2$ is topologically conjugate to a Kleinian action, using an argument similar to \cite[\S 5]{KapovichKleiner00}. 
If $\boundary(H,\mathcal{Q})$ is a \Sierpinski\ carpet, then its peripheral circles are pairwise disjoint and null. Thus, by \cite[Prop.~2.3]{HPWpreprint}, they lie in finitely many $H$--orbits, and the stabilizer of each circle $C$ is quasiconvex with limit set $C$. Choose a set of representatives $\{ K_1,K_2,....K_n \}$ for the conjugacy classes of circle stabilizers. Let $DH$ denote the double of $H$ over $\{K_i\}$; that is, $DH$ is the fundamental group of the graph of groups with two vertex groups, each isomorphic to $H$, and $n$ edges connecting them, joining the common copies of the $K_i$'s.
By Dahmani's combination theorem \cite{Dahmani03Combination}, the double $DH$ is hyperbolic relative to virtually $\mathbb{Z}^2$ subgroups and by the argument of \cite[Cor.~1.2]{TshishikuWalsh20}, its Bowditch boundary is $S^2$. (See \cite{KapovichKleiner00} for the word hyperbolic case). 
Thus, by the relative Cannon conjecture, $DH$ acts properly, isometrically, and geometrically finitely on $\mathbb{H}^3$.  
The Kleinian action must have the same maximal parabolic subgroups as the Bowditch boundary of $DH$, so the limit set of this Kleinian action is $\widehat{\C}$, and there is a $DH$--equivariant homeomorphism $S^2 \to \widehat{\C}$ conjugating the action of $DH$ on its boundary to the Kleinian action. Restricting to the limit set of $H$, we conclude that $H$ is topologically conjugate to a geometrically finite group of isometries of $\mathbb{H}^3$. 
\end{proof}

\begin{lem}
\label{lem:Sierpinski}
The \Sierpinski\ carpet $\mathcal{S}$ contains an embedded \Sierpinski\ carpet $\mathcal{S}'$ that is disjoint from every peripheral circle of $\mathcal{S}$.
\end{lem}

\begin{proof}
Let $M^3$ be a compact hyperbolic $3$--manifold with three totally geodesic boundary components $A$,$B$,  and $G$.
Take four $3$--manifolds $M_0, M_1$, $M_2$, and $M_3$, each isometric to $M$, and identify $A_0=A_1$, \ $B_0=B_2$, and $C_0=C_3$ to form a hyperbolic manifold with boundary that has $M_0$ in its interior.
The limit sets of their universal covers provide the required carpets.

Alternatively, collapse each peripheral circle of $\mathcal{S}$ to a point to get a $2$--sphere. The image of the set of peripheral circles is a countable set $D \subset S^2$.
A subcarpet disjoint from $D$ exists by the following theorem: For every subset $S$ of $\R^n$ with empty interior and any countable set $D$, there is a subset $S'$ homeomorphic to $S$ whose closure lies in $\R^n \setminus D$ \cite[Thm.~1.8.9]{Engelking_Dimensions}.
\end{proof}

\subsection{Assembling the hierarchy}
\label{sec:PuttingTogether}

We now prove several of our main results.

\begin{proof}[Proof of Theorem~\ref{thm:generalcase}] 
If some members of $\mathcal{P}$ are not virtually abelian, we express $\mathcal{P}$ as $\mathcal{Q} \sqcup \mathcal{P}_h$, where $\mathcal{P}_h$ is the family of all peripheral subgroups that are not virtually abelian. Such subgroups are either cocompact Fuchsian groups or virtually free but not virtually cyclic. Theorem~\ref{thm:CoveringConvergence} produces an action of $G$ on $S^2$ covering the original action, in which the action of each member of $\mathcal{P}_h$ is blown up to a Fuchsian action on a closed disc.

We now assume that all members of $\mathcal{P}$ are virtually abelian.
Proposition~\ref{prop:GoingDown} gives a finite hierarchy of the group $G$, which we can write as a rooted tree of height $h$, where each operation in the description of the hierarchy in Section~\ref{sec:goingdowntop} moves to a lower level.  The vertex group at the root is $G$.  Each vertex one level below the root is either the result of unpinching (Step~\ref{unpinchZ}), in which case there is only one vertex at this level, or each vertex is a vertex group of a splitting of $G$ corresponding to one of Steps \ref{splitfinite}, \ref{splitperif} or~\ref{splitZ}.   Each terminal vertex is topologically conjugate to a geometrically finite Kleinian group by Proposition~\ref{prop:GoingDown}. 

Thus, we know the result is true for groups of height $0$.
We proceed by induction on the height $h$ of the hierarchy.
If $(G,\mathcal{P})$ has at least one $2$--ended peripheral subgroup it has a single descendant $(G,\mathcal{Q})$ obtained by removing all $2$--ended groups from $\mathcal{P}$ (Step~\ref{unpinchZ} of hierarchy). By induction, $(G,\mathcal{Q})$, which is at height $h-1$, is conjugate to a geometrically finite Kleinian group.
By the inductive hypothesis, the action of $(G,\mathcal{Q})$ is topologically conjugate to a geometrically finite Kleinian action.  According to Proposition~\ref{prop:unpinchpinch}, the original action of $(G, \mathcal{P})$ is also topologically conjugate to a Kleinian group.

Otherwise all peripheral subgroups of $(G,\mathcal{P})$ are virtually $\Z^2$.
If $\boundary(G,\mathcal{P})$ is disconnected (that is, if the operation going down was Step~\ref{splitfinite}), apply Proposition~\ref{prop:FiniteDecomposition}.
If it is connected and has a cut point (that is, the operation going down was Step~\ref{splitZ}), apply Proposition~\ref{prop:LoxodromicDecomposition}.
If it is connected with no cut point but contains a cut pair (that is, the operation going down was Step~\ref{splitperif}), then apply Proposition~\ref{prop:ParabolicDecomposition}.
In each case, we are given a graph of groups whose vertex groups are conjugate to geometrically finite Kleinian groups, and the corresponding proposition shows that $(G,\mathcal{P})$ is also topologically conjugate to a geometrically finite Kleinian group.
\end{proof} 

\begin{proof}[Proof of Theorem~\ref{thm:puttogether}]
If some vertex group has boundary $S^2$, then the splitting is trivial and there is nothing to prove. 
Without loss of generality, we assume the limit set of each vertex group is not $S^2$.
Each vertex group of the given graph of groups is virtually compact special by Wise's Virtually Compact Special Theorem (Theorem~\ref{thm:VCS}). Therefore, $G$ is residually finite by a result of Huang--Wise (see Theorem~\ref{thm:HW_Stature}).

We claim that $G$ contains a torsion-free finite index subgroup $G'$ such that the induced graph of groups splitting of $G'$ is regular in the sense of Definition~\ref{defn:Regular}.
Indeed, each vertex group---being Kleinian---is virtually torsion free, so by residual finiteness, there is a finite index subgroup of $G$ that has this property. 
For regularity, consider a $2$--ended vertex group.  The inclusion of each adjacent edge group is finite index, and there are finitely many such edges.  So, by passing to a further finite-index subgroup, we can make these all isomorphisms. Call this finite index subgroup $G'$.  

As above, $G'$ inherits a graph of groups structure where each edge group is $2$--ended.    Each vertex group $H_v'$ is the fundamental group of a hyperbolic manifold (some of these may be solid tori).  The edge groups correspond to embedded curves in the conformal boundary.  Gluing these pieces together as in Proposition~\ref{prop:ParedCombination} gives a pared $3$--manifold $M$ with nonempty boundary whose fundamental group is $G'$. The $3$--manifold $M$ is hyperbolic by the Hyperbolization Theorem (see Theorem~\ref{thm:Hyperbolization}), so $G$ is virtually Kleinian.
\end{proof} 

\begin{proof}[Proof of Theorem~\ref{thm:main}]
By hypothesis $(G,\mathcal{P})$ has no splittings over a finite or parabolic subgroup.
Consider the JSJ decomposition over $2$--ended subgroups. Each non-elementary vertex group $G_v$ of this splitting is relatively hyperbolic with respect to the collection $\mathcal{S}_v$, consisting of all infinite intersections of $G_v$ with members of $\mathcal{P}$ together with the subgroups that stabilize the edges adjacent to $v$.
According to \cite[Thm.~8.8]{HruskaWalsh}, the action of each rigid vertex group $G_v$, on its planar limit set $\boundary(G_v,\mathcal{S}_v)$ extends to a geometrically finite convergence group action on $S^2$. 
Similarly, the action of each quadratically hanging vertex group on the circle $\boundary(G_v,\mathcal{S}_v)$, also extends to such a convergence group action on $S^2$, since it is virtually Fuchsian. For each $2$--ended vertex group $G_v$, choose a Kleinian action that is virtually loxodromic.
Note that the action of each vertex group $G_v$ on $S^2$ could have a finite kernel $F_v$.
According to Theorem~\ref{thm:generalcase}, if $G_v$ is nonelementary, the action of the quotient group $G_v/F_v$ on $S^2$ is covered by a Kleinian action on $\widehat{\C}$ in which each adjacent $2$--ended edge group $G_e$ acts parabolically.
Since the abstract group $G_v$ maps onto a Kleinian group with finite kernel, $G_v$ has a finite index torsion-free subgroup $H_v$  by \cite[Thm.~1.3]{HaissinskyLecuire_3manifold}.

Thus, $H_v$ has a faithful Kleinian action in which each edge group acts parabolically. There is an associated peripheral structure $\mathcal{P}_v \sqcup \mathcal{Q}_v$ containing only virtually abelian groups such that $\mathcal{Q}_v$ is the family of stabilizers of edges adjacent to $v$. This peripheral structure may be different from $\mathcal{S}_v$, since the nonelementary hyperbolic surface groups have been removed in the covering in \ref{thm:generalcase}.   Note that each member of $\mathcal{Q}_v$ is two-ended. Let $M_v = \Hyp^3/H_v$ be the associated hyperbolic $3$--manifold with pared structure $(M_v, P_v \cup Q_v)$, where $Q_v$ is a disjoint collection of incompressible annuli corresponding to the $2$--ended groups of $\mathcal{Q}_v$. Removing $Q_v$ from the pared structure gives a new pared structure $(M_v,P_v)$ that is again hyperbolic by the Hyperbolization Theorem (see Theorem~\ref{thm:Hyperbolization}).

Each two-ended parabolic subgroup in $\mathcal{Q}_v$ corresponds to a rank-one cusp of $(M_v, P_v \cup Q_v)$, and in the new pared structure $(M_v, P_v)$, these are embedded annuli in the conformal boundary of $(M_, P_v)$. Applying Theorem~\ref{thm:puttogether}, the group $G$ is virtually the fundamental group of a hyperbolic 3-manifold. 
\end{proof} 

\section{Applications}
\label{sec:Applications}

We now give some applications to the conjecture of Martin and Skora, and to variations of the Cannon conjecture. 

\begin{conj}[Cannon Conjecture] \label{conj:Cannon}
Let $G$ be a word hyperbolic group whose Gromov boundary is homeomorphic to $S^2$. Then $G$ acts properly, isometrically, and cocompactly on $\Hyp^3$.
\end{conj}

\begin{conj}[Relative Cannon Conjecture]  \label{conj:relativecannon}
Let $(G,\mathcal{P})$ be a relatively hyperbolic group pair whose Bowditch boundary is homeomorphic to $S^2$. Then $G$ acts properly, isometrically, and geometrically finitely on $\Hyp^3$.
\end{conj}

In the Relative Cannon Conjecture, the members of $\mathcal{P}$ do not need to be virtually abelian.  Thus, the members of $\mathcal{P}$ may not all act parabolically in the corresponding Kleinian representation.

\begin{conj}[GF Martin--Skora Conjecture] 
\label{conj:GFMartinSkora}
If $G < \Homeo(S^2)$ is a geometrically finite convergence group, then $G$ is covered by a Kleinian group. 
\end{conj} 

For an example of a convergence group on $S^2$ that properly covers another convergence group, see Definition~\ref{def:covering}.  

\begin{thm}
\label{thm:equivalent} 
The three conjectures listed above are equivalent within the class of all torsion-free groups. That is, if any of the conjectures is true for all torsion-free groups $G$, then the others are as well.
\end{thm} 

\begin{proof}
\textbf{Claim:}
Conjecture~\ref{conj:GFMartinSkora} implies Conjecture~\ref{conj:Cannon} for torsion-free groups.

Assume the GF Martin--Skora Conjecture for torsion-free groups.
If $G$ is word hyperbolic with Gromov boundary homeomorphic to $S^2$, and $G$ is torsion free, then $G$ acts faithfully on $S^2$.  The GF Martin--Skora Conjecture implies that $G$ is covered by a Kleinian group $\Gamma$ abstractly isomorphic to $G$. According to \cite[Cor.~1.4]{BestvinaMess91}, the torsion-free group $\Gamma$ has cohomological dimension $3$, so the manifold $\Hyp^3 / \Gamma$ is a closed hyperbolic $3$--manifold. In particular, $\Gamma\cong G$ acts properly, isometrically, and cocompactly on $\Hyp^3$.

\textbf{Claim:} Conjecture~\ref{conj:Cannon} implies Conjecture~\ref{conj:relativecannon} for torsion-free groups.

It follows from the proof in Groves-Manning--Sisto \cite[Cor.~1.4]{GrovesManningSisto19} that the Cannon conjecture for torsion-free groups implies the relative Cannon conjecture for torsion-free relatively hyperbolic pairs $(G,\mathcal{P})$ if all peripheral subgroups are $\Z^2$.
Indeed, the proof involves a sequence of Dehn fillings of $G$, each with torsion-free peripheral subgroups.  By \cite[Thm.~4.1]{GrovesManning_Elementary}, since the given group $G$ is torsion free, so are all the filled groups arising in their proof.  
The Cannon conjecture is applied to the torsion-free groups in this sequence. Hence, each is a $3$--manifold group. We conclude that the limit of the sequence of fillings is a Kleinian representation of $G$, realizing $G$ as the fundamental group of a cusped hyperbolic $3$--manifold.  

Now suppose $(G,\mathcal{P})$ is relatively hyperbolic with Bowditch boundary $\partial (G, \mathcal{P}) = S^2$, and that $G$ is torsion free. Then $G$ contains a finite index subgroup $G'$ that is orientation preserving on $S^2$. Let $\mathcal{P}'$ denote the induced set of peripheral subgroups of $G'$. Then $\partial(G', \mathcal{P}')$ is also $S^2$, and each member of $\mathcal{P}'$ is the fundamental group of a closed orientable surface (see \cite[Thm.~0.3]{Dahmaniparabolic} or \cite[Cor.~3.2]{HruskaWalsh}). If $\mathcal{P}'$ contains no hyperbolic surface groups, then each member of $\mathcal{P}'$ is free abelian. If the Cannon conjecture for torsion-free groups holds, then by the argument above using \cite[Cor.~1.4]{GrovesManningSisto19}, $G$ is the fundamental group of a cusped hyperbolic 3-manifold and hence satisfies the relative Cannon conjecture, Conjecture~\ref{conj:relativecannon}.

Suppose that some of the peripheral groups are hyperbolic. Let $\mathcal{P}'_h$ be the family of word hyperbolic peripheral subgroups in $\mathcal{P}'$, which are all the fundamental groups of closed hyperbolic orientable surfaces. 
According to Theorem~\ref{thm:unpinch}, the Bowditch boundary $M=\partial(G',\mathcal{P'} \setminus \mathcal{P}'_h)$ embeds in $S^2$.  According to \cite[Lem.~12]{TshishikuWalsh20}, this planar set $M$ is homeomorphic to the \Sierpinski\ carpet if $\mathcal{P}'_h$ is nonempty. The double $DG'$ of $G'$ across the subgroups $\mathcal{P}'_h$ is relatively hyperbolic with all peripheral subgroups isomorphic to $\Z^2$ by Dahmani's combination theorem \cite[Thm.~0.1]{Dahmani03Combination}. 
It follows that the boundary of the double $DG'$ is homeomorphic to $S^2$, as explained in the proof of \cite[Cor.~1.2]{TshishikuWalsh20} (see also \cite{KapovichKleiner00} for the word hyperbolic case).

Since the parabolic subgroups of $DG'$ acting as a convergence group on $S^2$ are free abelian, we again cite \cite[Cor.~1.4]{GrovesManningSisto19} as above, which implies (under our assumption of the Cannon conjecture for torsion-free groups) that $DG'$ acts properly and isometrically on $\Hyp^3$ with finite volume quotient.
The subgroup $G'$ must be geometrically finite in $DG'$ since it is not a virtual fiber. It follows from Theorem~\ref{thm:virtuallyKleinian} that $G$ itself has such an action. Thus, $G$ acts properly, isometrically, and geometrically finitely on $\mathbb{H}^3$. 

\textbf{Claim:} Conjecture~\ref{conj:relativecannon} implies Conjecture~\ref{conj:GFMartinSkora} for torsion-free groups.

This implication is given by Theorem~\ref{thm:generalcase}.
\end{proof}

\begin{proof}[Proof of Theorem~\ref{thm:CannonEquivalent}]
This result is one of the parts of Theorem~\ref{thm:equivalent}. 
\end{proof}

\bibliographystyle{alpha}
\bibliography{citations.bib}
\end{document}